\documentclass[12pt]{amsart}
\usepackage{amssymb, todonotes, amsmath}

\usepackage[margin=2.5cm, a4paper]{geometry}
\usepackage{tikz}
\usetikzlibrary{matrix}
\usetikzlibrary{chains}

\newtheorem{thm}{Theorem}[section]
\newtheorem{prop}[thm]{Proposition} 
\newtheorem{cor}[thm]{Corollary} 

\newtheorem{lem}[thm]{Lemma}
\newtheorem{rem}[thm]{Remark}
\newtheorem{defn}[thm]{Definition}

\newtheorem{example}[thm]{Example}
\numberwithin{equation}{section}

\renewcommand{\MR}[1]{}

\newcommand{\cald}{\mathcal{D}}
\newcommand{\cala}{\mathcal{A}}
\newcommand{\calc}{\mathcal{C}}

\newcommand{\calf}{\mathcal{F}}
\newcommand{\calt}{\mathcal{T}}
\newcommand{\calk}{\mathcal{K}}
\newcommand{\calj}{\mathcal{J}}
\newcommand{\call}{\mathcal{L}}

\begin{document}
	
\title[Weighted Cuntz-Krieger algebras]%
{Weighted Cuntz-krieger algebras}
%
\author[L. Helmer and B. Solel]{Leonid~Helmer \, and \, Baruch~Solel}
\thanks{2010 {\it Mathematics Subject Classification.} 46L05, 47L80, 46L08, 46L35, 46L89.  }
\thanks{{\it key words and phrases.} Directed graph, Graph algebras, Cuntz-Krieger algebras, Weighted shift, Simplicity, Cuntz-Pimsner
	algebra, $C^*$-Correspondence, Fock space, $C^*$-algebra, Gauge-invariant ideals}
\address{Department of Mathematics, Ben Gurion University, Beer Sheva, Israel} \email{leonihe@gmail.com}
\address{Department of Mathematics, Technion - Israel Institute of Technology, Haifa 32000, Israel}
\email{mabaruch@technion.ac.il}

\date{}

%


\begin{abstract}
	Let $E$ be a finite directed graph with no sources or sinks and write $X_E$ for the graph correspondence.
	We study the $C^*$-algebra $C^*(E,Z):=\mathcal{T}(X_E,Z)/\mathcal{K}$ where $\mathcal{T}(X_E,Z)$ is the $C^*$-algebra generated by weighted shifts on the Fock correspondence $\mathcal{F}(X_E)$  given by a weight sequence $\{Z_k\}$ of operators $Z_k\in \mathcal{L}(X_{E^{k}})$ and $\mathcal{K}$ is the algebra of compact operators on the Fock correspondence. If $Z_k=I$ for every $k$, $C^*(E,Z)$ is the Cuntz-Krieger algebra associated with the graph $E$. 
	
	We show that $C^*(E,Z)$ can be realized as  a Cuntz-Pimsner algebra and use a result of Schweizer to find conditions for the algebra $C^*(E,Z)$ to be simple. We also analyse the gauge-invariant ideals of $C^*(E,Z)$ using a result of Katsura and conditions that generalize the conditions of subsets of $E^0$ (the vertices of $E$) to be hereditary or saturated. 
	
	As an example, we discuss in some details the case where $E$ is a cycle.
\end{abstract}

\maketitle

\section{Introduction }

In \cite{MuS16} the second author, with P. Muhly, introduced and studied algebras of weighted shifts on the Fock space associated with a correspondence. This is a far reaching generalization of the classical weighted shift (on $\ell_2$). The emphasis there was on the nonself adjoint operator algebras generated by such shifts (these algebras are called weighted Hardy algebras). In the last section of that paper they studied the $C^*$-algebras generated by weighted shifts on the Fock correspondence associated with the correspondence $_{\alpha}M$ where $M$ is a von Neumann algebra and $\alpha$ is an automorphism on $M$. Such an algebra was referred to as a weighted crossed product and, if the weights were trivial, one gets the usual, unweighted, crossed product (of $M$ by $\alpha$). It was shown there that the weighted crossed product could be presented as an unweighted crossed product (of another $C^*$-algebra by a certain automorphism).

This generalizes a result of O'Donovan (\cite{OD75}) that shows that the $C^*$-algebra generated by a single weighted shift (on $\ell_2$) modulo the compact operators on $\ell_2$, is isomorphic to the crossed product of a certain commutative $C^*$-algebra by an action of $\mathbb{Z}$.

In \cite{HS} we considered the $C^*$-algebra generated by $d$ weighted shifts modulo the compact operators. More precisely, such algebra is generated by weighted shifts on the Fock space $\mathcal{F}(\mathbb{C}^d)$ (with weights given by a sequence $\{Z_k\}$ of $d^k\times d^k$ matrices) modulo the compact operators on the Fock space. If $Z_k=I$ for every $k$, we end up with the Cuntz algebra $O_d$. In the general case, we referred to such a $C^*$-algebra as a weighted Cuntz algebra. We showed there that every weighted Cuntz algebra can be presented as a Cuntz-Pimsner algebra and we used this fact to study the simplicity of the algebra.

Here we explore the case where the Fock space $\mathcal{F}(\mathbb{C}^d)$ is replaced by the Fock correspondence associated with a directed graph.

More precisely, we fix a finite directed graph $E$ with no sources or sinks and consider the graph correspondence $X_E$. The algebra we study are subalgebras of $\mathcal{L}(\mathcal{F}(X_E))/K(\mathcal{F}(X_E))$ that are generated by a weighted shift on the Fock correspondence $\mathcal{F}(X_E)$ modulo $K(\mathcal{F}(X_E))$. The weights are given by a sequence $\{Z_k\}$ of positive, adjointable operators on $\{(X_E)^{\otimes k}\}$ as in \cite{MuS16} (see Definition~\ref{weights}). If $Z_k=I$ for every $k$, we get the Cuntz-Krieger algebra $C^*(E)$ introduced by Cuntz and Krieger in \cite{CK}. For a general $Z:=\{Z_k\}$, we write $C^*(E,Z) $ for the algebra and refer to it as the weighted Cuntz-Krieger algebra associated with $E$ and $Z$.

In Theorem~\ref{isomorphism} we show that the weighted Cuntz-Krieger algebra $C^*(E,Z)$ is isomorphic to a Cuntz-Pimsner algebra $O(q(F),q(\mathcal{D}))$ associated with a $C^*$-correspondence $q(F)$ over a $C^*$-algebra $q(\mathcal{D})$. Then, at least in principle, one can apply the theory of Cuntz-Pimsner algebras to study the algebra $C^*(E,Z)$. 

We then use this approach to study the simplicity of the algebra and the collection of all its gauge-invariant ideals.

The problem with this approach (as the one we had in \cite{HS}) is that dealing with a general sequence of weights can be quite complicated and we had to impose an additional condition. The condition we impose is that the sequence $\{Z_k\}$ is essentially periodic of period $p$ for some natural number $p$. We refer to it as Condition A(p) (see Lemma~\ref{uz} and the discussion preceding it). Note that the (unweighted) Cuntz-Krieger algebra $C^*(E)$ satisfies this condition with $p=1$.

In the rest of the paper we assume that Condition A(p) holds for some natural number $p$.

In order to study the ideals in $C^*(E,Z)$ we first study the algebra $q(\mathcal{D})$ and its ideals. We define an increasing sequence of sub $C^*$-algebras, $\mathcal{C}_0\subseteq \mathcal{C}_1 \subseteq \ldots $ such that $q(\mathcal{D})=\overline{\cup_n \mathcal{C}_n}$. For every $v\in E^0$, we define a certain ``corner" of $\mathcal{C}_0$, denoted $\mathcal{C}_v$, and we show that each $\mathcal{C}_n$ is isomorphic to a $C^*$-algebra, $\mathcal{A}_n$, given, roughly, by a bundle over a subset of $\Gamma_n\times \Gamma_n$ (where $\Gamma_n$ is the set of all paths of length $(n+1)p-1$ ) and each fibre is one of the algebras $\{\mathcal{C}_v\}$ (see Proposition~\ref{cn} for the precise statement). It follows that $q(\mathcal{D})$ is a direct limit of $\{\mathcal{A}_n\}$ (Corollary~\ref{dirlim} ).

Given a (closed, two sided) ideal $J$ in $q(\mathcal{D})$ it is equal to $\overline{\cup_n (J\cap \mathcal{C}_n)}$ and, as we prove in Proposition~\ref{ideals}, to each $J\cap \mathcal{C}_n$ we can associate a family $\{(J\cap \mathcal{C}_n)_v\}_{v\in E^0}$ that fully describe it and each $(J\cap \mathcal{C}_n)_v$ is an ideal in $\mathcal{C}_v$.

Using this analysis, we are able, in Theorem~\ref{simplicity}, to find conditions for the algebra $C^*(E,Z)$ to be simple. For this we use the characterization of simplicity for Cuntz-Pimsner algebras proved by J. Schweizer in \cite{Sch01} (see Theorem~\ref{simpleCP}).

To describe the set of all the gauge-invariant ideals we use a result of Katsura (\cite{Ka2}) that proves that there is a bijection between the set of all gauge-invariant ideals in a Cuntz-Pimsner algebra $O(X,A)$ and the set of all $O$-pairs (see Definition~\ref{Opair} for the definition of an $O$-pair and Theorem~\ref{Opairsinvideals} for the statement of his result). It turns out that the $O$-pairs in our case are all pairs of the form $(J,q(\mathcal{D}))$ where $J\subseteq q(\mathcal{D})$ is a fully invariant ideal (in the sense of Definition~\ref{finv}). Thus, there is a bijection between the set of all gauge-invariant ideals in $C^*(E,Z)$ and the set of all fully invariant ideals in $q(\mathcal{D})$ (Corollary~\ref{finvideals}).

In the unweighted case (Cuntz-Krieger algebras) it is known (\cite[Theorem 4.1 (a)]{BPRS}) that the gauge-invariant ideals of $C^*(E)$ correspond to subsets of $E^0$ that are hereditary and saturated (see Definition~\ref{sather}). Here, in the weighted case, we consider families $\{J_v\}_{v\in E^0}$ of ideals $J_v\subseteq \mathcal{C}_v$ and define, for such a family, two conditions, called (H) (for Hereditary) and (S) (for Saturated) and prove in Theorem~\ref{gaugeinvid} that there is a bijection (explicitely given) between the collection of all such families that satisfy (H) and (S) and the collection of the fully invariant ideals in $q(\mathcal{D})$ (and, therefore, also the collection of all the gauge-invariant ideals in $C^*(E,Z)$).

 In Section 5 we discuss the unweighted case and show that, indeed, our general theorem (Theorem~\ref{gaugeinvid}) can be used to give a new proof of the result of Bates, Pask, Raeburn and Szymanski mentioned above.
 
 Another case that we study in some detail is when $E$ is a directed cycle. In this case, the unweighted algebra $C^*(E)$ is not simple but it has no non trivial gauge-invariant ideals. We show in Section 6 that, for some choice of weighted sequence $Z$, the weighted algebra $C^*(E,Z)$ does have non trivial gauge-invariant ideals.

\section{The weighted Cuntz-Krieger algebras as Cuntz-Pimsner algebras}

Let $E=(E^0, E^1, s,r)$ be a finite directed graph where $E^0$ and $E^1$ are finite sets of vertices and edges respectively, and $s:E^1\rightarrow E^0$ and $r:E^1\rightarrow E^0$ the source and range maps respectively.
The vertex $v\in E^0$ is called a source if $r^{-1}(v)=\emptyset$ and $v$ is a sink if $s^{-1}(v)=\emptyset$. 

We assume throughout that $E$ has no sinks. Starting in Section 4, we also assume that $E$ has no sources.

In this section, we allow sources.

Given a finite graph $E=(E^0, E^1, s,r)$, then $E$ defines a graph correspondence as follows.
Let $A:=c(E^0)$ be the $C^*$-algebra of all functions defined on $E^0$ with the $\sup$-norm.
Thus $A$ is a finite dimensional $C^*$-algebra.
Write $\delta_v$ for the function which is $1$ at $v\in E^0$ and $0$ elsewhere. Then every $a\in A$ can be written as 
$a=\sum_{i=1}^{n}\alpha_i\delta_{v_i}$, $\alpha_i\in \mathbb{C}$, and $a(v_i)=\alpha_i$.

Let $c(E^1)$ be a set of all functions defined on $E^1$, then $c(E^1)$ is naturally a $A$-bimodule with the left and right actions defined by
$(a\cdot \xi)(e)=a(r(e))\xi(e)$ and $(\xi\cdot a)(e)=\xi(e)a(s(e))$ respectively, where $\xi\in c(E^1)$, $a\in A$ and $e\in E^1$. We write $\varphi_E(a)\xi$ for the left action. When $E$ has no sources the left action $\varphi_E$ is faithful.
For the $A$-valued inner product we set
$$\langle \xi,\eta\rangle(v):=\sum_{e\in s^{-1}(v)}\overline{\xi(e)}\eta(e),$$
where  $\xi,\eta\in c(E^1)$, $e\in E^1$, $v\in E^0$.
By separation and completion we obtain a $C^*$-correspondence over $A$. We denote this correspondence by $X_E$. When $E$ has no sinks $X_E$ is a full correspondence \cite[Proposition 8.8]{Rae}. 

Let $\calk(X_E)$ be the algebra of generalized compact operators on $X_E$:
$$\calk(X_E)=\overline{span}\{\theta_{\xi,\eta}:\xi,\eta\in X_E\},$$
where $\theta_{\xi,\eta}$ defined by
$$\theta_{\xi,\eta}(\zeta)=\xi\langle \eta,\zeta\rangle$$
If we write $\delta_e$ for the function on $E^1$ which is $1$ at $e\in E^1$ and $0$ elsewhere, then every $\xi\in E^1$ has the form $\xi=\sum_{e\in E^1} \alpha_e\delta_{e} $ where $\xi(e)=\alpha_e$.
Then $$\calk(X_E)=\overline{span}\{\theta_{\delta_{\alpha},\delta_{\beta}}:\alpha,\beta \in E^1\}.$$
It follows from \cite[Proposition 8.8]{Rae} and the finiteness of $E$ that  $\varphi_E(\delta_v)\in \calk(X_E)$, and that 
$$\varphi_E(\delta_v)=\sum_{\{e:v=r(e)\}}\theta_{\delta_e,\delta_e}$$ and $\varphi_E(\delta_v)=0$ if $v$ is a source.
It follows that $\calk(X_E)$ is a unital algebra with the identity $I=\sum_{e\in E^1}\varphi_E(\delta_{e})$, hence $\calk(X_E)=\call(X_E)$.

For $k\in \mathbb{N}$, write $E^k$ for the set of all paths in $E$ of length $k$. For such a path $\alpha=e_1 e_2 \cdots e_k$ ($e_i\in E^1$ and $s(e_j)=r(e_{j+1})$ for $1\leq j \leq k-1$) we write $s(\alpha)=s(e_k)$, $r(\alpha)=r(e_1)$ and $|\alpha|=k$. This allows us to view $E^k$ as a graph and then the correspondence $X_{E^k}$ (over $A$) is well defined. Moreover, it is not hard to check that, for $k,m$ we have an isomorphism of correspondences
\begin{equation}
	X_{E^k}\otimes_A X_{E^m}\cong X_{E^{k+m}}
\end{equation}
where the isomorphism maps $\delta_{\alpha}\otimes_A \delta_{\beta}$ to $\delta_{\alpha \beta}$  (note that $\delta_{\alpha}\otimes_A \delta_{\beta}$ is non zero only if $r(\beta)=s(\alpha)$) . (See \cite[Lemma 5.1]{MuS99}). In particular, $X_{E^k}\cong X_E^{\otimes k}$. Note that $X_{E^k}$ is spanned by $\{\delta_{\alpha}:|\alpha|=k\}$ and we shall often identify  $X_{E^k}$ and $ X_E^{\otimes k}$. For $k=0$ we write $X_E^{\otimes 0}=A$.

This allows us to define the full Fock correspondence $\calf(X_E)=\oplus_{k\geq 0}X_E^{\otimes k}$. This is a correspondence over $A$. 

We write $Q_k$ for the projection, in $\mathcal{L}(\mathcal{F}(X_E))$, onto the $k$th summand $X_E^{\otimes k}$.

Using our assumption that the graph is finite, we get the following.

\begin{lem}\label{K}
	Each $Q_k$, $k\geq 0$, is in $\calk(\calf(X_E))$ and, for $T\in \call(\calf(X_E))$, $T\in \calk(\calf(X_E))$ if and only if 
	$$\lim_{k\rightarrow \infty} \| Q_k T Q_k\| =0.$$	
	
\end{lem}
\begin{proof}
	It is easy to check that $Q_k=\Sigma_{|\alpha|=k} \theta_{\delta_{\alpha},\delta_{\alpha}}$ which lies in $\calk(\calf(X_E))$ because $E$ is finite. Since $\calk(\calf(X_E))$ is a closed ideal, it is left to show that, if 
	$T\in \calk(\calf(X_E))$ then $\lim_{k\rightarrow \infty} \| Q_k T Q_k\| =0$.
	
	For this it suffices to let $T$ be a generator $T=\theta_{\xi,\eta}$. Then, for every $k$, $Q_kTQ_k=\theta_{Q_k\xi,Q_k\eta}$ and, thus, for such $T$, $\|Q_kTQ_k\|\leq \|Q_k\xi\| \|Q_k\eta\| \rightarrow_k 0$ (as $\Sigma_k Q_k=I$). 
	
\end{proof}

For $\alpha\in E^1$ we define the shift (or creation) operator $S_{\alpha}$ on the Fock correspondence by
\begin{equation}
	S_{\alpha}(\xi_1 \otimes \cdots \xi_k)=\delta_{\alpha}\otimes \xi_1 \otimes \cdots \xi_k .
\end{equation}
Using the identification of $X_{E^k}$ and $ X_E^{\otimes k}$, we can write
\begin{equation}
	S_{\alpha}\delta_{\beta}=\delta_{\alpha \beta}
\end{equation}
for a path $\beta$ of length $k$.

For $a\in A$, we define the operator $\varphi_{\infty}(a)$ on the Fock correspondence by
\begin{equation}
	\varphi_{\infty}(a)(\xi_1 \otimes \cdots \xi_k)=(\varphi_E(a)\xi_1) \otimes \cdots \xi_k
\end{equation}
and, using the identification above,
\begin{equation}
	\varphi_{\infty}(\delta_v)\delta_{\beta}=\delta_{\beta}
\end{equation} if $r(\beta)=v$ and $0$ otherwise.
Both $S_{\alpha}$ and $	\varphi_{\infty}(\delta_v)$ (for $\alpha\in E^1$ and $v\in E^0$) are in $\mathcal{L}(\calf(X_E))$. 

Also note that 
$	\varphi_{\infty}(\delta_v)$ is a projection and we write $P_v$ for it. We also write $P_s$ for the sum of the projections associated with vertices that are sources. That is
\begin{equation}\label{Ps}
	P_s=\Sigma_{|r^{-1}(v)|=0} P_v.
\end{equation} Note that, if $v$ is a source, $P_v\delta_{\alpha}=0$ whenever $|\alpha|>0$. Thus $$P_s\leq Q_0$$ and $P_s\in \calk(\calf(X_E))$.

It will also be convenient to write, for every path $\alpha$ of length $k$, $\alpha=\alpha_1 \cdots \alpha_k$ (with $\alpha_i\in E^1$), $S_{\alpha}=S_{\alpha_1}\cdots S_{\alpha_k}$. Thus $S_{\alpha}\delta_{\beta}=\delta_{\alpha \beta}$.

For the adjoint, $S_{\alpha}^*$, it is easy to check that $S_{\alpha}^*\delta_{\beta}$ is different from $0$ only if we can write $\beta=\alpha \gamma$ (for some path $\gamma$) and, in this case, $S_{\alpha}^*\delta_{\beta}=\delta_{\gamma}$.

It follows that, for different paths $\alpha,\beta$ of the same length,
\begin{equation}
	S_{\alpha}^*S_{\beta}=0
\end{equation}and \begin{equation}
	S_{\alpha}^*S_{\alpha}=P_{s(\alpha)}
\end{equation}

In particular, each $S_{\alpha}$ is a partial isometry.

It is also easy to check that, for $k\geq 1$ and $v\in E^0$,
$$\Sigma_{|\alpha|=k, r(\alpha)=v} S_{\alpha}S_{\alpha}^*\leq P_v.$$
Thus $\{S_{\alpha}: |\alpha|=k\}$ is a Toeplitz-Cuntz-Krieger family.

In fact, for a vertex  $v$ that is not a source,
\begin{equation}\label{sumSonv}
	\Sigma_{|\alpha|=k, r(\alpha)=v} S_{\alpha}S_{\alpha}^*=P_v(I-\Sigma_{i=0}^{k-1} Q_i)
\end{equation} where $Q_i$ is the projection of the Fock correspondence onto the $i$'th summand. If $v$ is a source, the sum on the left is just $0$. Thus
\begin{equation}\label{sumS}
	\Sigma_{|\alpha|=k}S_{\alpha}S_{\alpha}^*=(I-\Sigma_{i=0}^{k-1} Q_i)P_s^{\perp}.
\end{equation}

Now, the Toeplitz algebra 
$\calt(X_E)$ is defined to be $C^*(\{S_{\alpha}: \alpha\in E^1\})$.

The following definition appeared in \cite{MuS16} in the context of $W^*$-correspondences and has been used by us in \cite{HS}.

\begin{defn}\label{weights}  A sequence $Z=\{Z_k\}_{k\geq 0}$ of operators $Z_k\in \call((X_E)^{\otimes k})\cap \varphi_{X_E^{\otimes k}}(A)'$  will be called a weight sequence  in case:
	
	1. $Z_0=I_A$,
	
	2. $\sup\|Z_k\|<\infty$,
	
	3.	each $Z_k$ is positive and invertible, and
	
	4. There is an $\epsilon >0$ such that $Z_k\geq \epsilon I$ for all $k\geq 1$ .
	
\end{defn}

Given a weight sequence it defines a weight operator $Z=diag(Z_1,Z_2,...) :\calf(X_E)\rightarrow \calf(X_E)$,
where $Z_k:
X_E^{\otimes k}\rightarrow X_E^{\otimes k}$. We write, for $\alpha\in E^1$, $W_{\alpha}=ZS_{\alpha}$ and refer to it as a weighted shift. The weighted Toeplitz algebra, $\mathcal{T}(X_E,Z)$, is the $C^*$-algebra generated by $\{W_{\alpha}: \alpha\in E^1\}$. It follows from our assumptions on the weight sequence that the hypotheses $\textbf{A}$ and $\textbf{B}$ from \cite[Section 6]{MuS16} are satisfied and (using  \cite[Proposition 6.5]{MuS16}) the Toeplitz algebra $\mathcal{T}(X_E)$ is a subalgebra of the the weighted Toepliz algebra $\mathcal{T}(X_E, Z)=C^*(W_{\alpha}:\alpha\in E^1)$. It also follows that $\mathcal{T}(X_E, Z)$ contains the algebra $\mathcal{K}(\calf(X_E)$ of the compact operators on the Fock correspondence.

The following lemma is a consequence of the inclusion $\calt(X_E)\subseteq \calt(X_E, Z)$. 
\begin{lem}\label{T}
	The operators $Q_k$ ($0\leq k$), $S_{\alpha}$ ($\alpha\in E^1$) and $Z$ are all contained in $\mathcal{T}(X_E, Z)$. So, $\mathcal{T}(X_E, Z)$ is the $C^*$-algebra generated by $\{I,S_{\alpha},Z: \alpha\in E^1\}$.
\end{lem}
\begin{proof} The inclusion $\calt(X_E)\subseteq \calt(X_E, Z)$ implies that
	$S_{\alpha} \in \mathcal{T}(X_E, Z)$. Hence, $Q_0=I-\Sigma S_{\alpha}S_{\alpha}^*$ and $Q_k=\Sigma_{|\alpha|=k} S_{\alpha}P_0S_{\alpha}^*$ are in $\calt(X_E, Z)$. For $Z$, note that $ZQ_0=Q_0\in \mathcal{T}(X_E, Z)$ and $ZQ_0^{\perp}=\Sigma_{\alpha\in E^1} (ZS_{\alpha})S_{\alpha}^* \in \mathcal{T}(X_E, Z)$.
\end{proof}
Write $W_t:=\Sigma_{n=0}^{\infty}e^{int}Q_n$
and $\gamma_t=Ad (W_t)$ for $t\in \mathbb{R}$ to get the gauge automorphism group on $\call(\mathcal{F}(X_E))$. 



Next, we write
$$\mathcal{D}:=\{X\in \mathcal{T}(X_E,Z):\; \gamma_t(X)=X\; \forall t \;\} $$ and 
$$ F:=\{X\in \mathcal{T}(X_E,Z):\; \gamma_t(X)=e^{it}X\; \forall t\;\} .$$
For simplicity, we shall often write $\mathcal{T}$ for $\mathcal{T}(X_E, Z)$.

\begin{lem}
	\begin{enumerate}
		\item[(1)] $F$ is a $C^*$-correspondence over the $C^*$-algebra $\mathcal{D}$ where the left and right actions are defined by multiplication and the $\mathcal{D}$-valued inner product is $\langle T,S\rangle=T^*S$.
		\item[(2)] $\mathcal{D}=\{T\in \mathcal{T}: TQ_k=Q_kT , \;k\geq 0 \}\subseteq \Sigma^{\oplus}_{k\geq 0} Q_k\mathcal{T}Q_k$ and $F=\{T\in \mathcal{T}: TQ_{k}=Q_{k+1}T , \;k\geq 0 \}\subseteq \Sigma^{\oplus}_{k\geq 0} Q_{k+1}\mathcal{T}Q_k$.
		\item[(3)] Writing $\varphi_F$ for the left action of $\mathcal{D}$ on $F$, we get $\ker(\varphi_F)=\mathcal{D}Q_0=Q_0\mathcal{T}Q_0$.
		\item[(4)] Write $K(F)$ for the algebra of generalized compact operators on $F$ then $K(F)=\varphi_F(\mathcal{D})=\mathcal{L}(F)$
	\end{enumerate}
\end{lem}
\begin{proof}
	The proof of the parts (1)-(3) is straightforward and is omitted.
	For part (4), recall that $K(F)\subseteq \mathcal{L}(F)$ is the norm closed ideal generated by the operators $\theta_{\xi_1,\xi_2}$ (for $\xi_1,\xi_2 \in F$) where $$\theta_{\xi_1,\xi_2}\xi_3=\xi_1 \langle \xi_2,\xi_3\rangle=\xi_1\xi_2^*\xi_3.$$
	But, then, $\theta_{\xi_1,\xi_2}=\varphi_F(\xi_1\xi_2^*)\in \varphi_F(\mathcal{D})$. Thus, $K(F)\subseteq \varphi_F(\mathcal{D})$. 
	Since, $\call({X_E}$) is a unital algebra and $\varphi_F$ is unital we obtain
	$$I=\varphi_F(I)=\varphi_F(Q_0^{\perp})=\varphi_F(\Sigma_{|\alpha|=1} S_{\alpha}S_{\alpha}^*)=\Sigma_{|\alpha|=1} \theta_{S_{\alpha},S_{\alpha}}\in K(F).$$
\end{proof}

Write $\mathcal{K}$ for the algebra of compact operators on the Fock space $\mathcal{F}(X_E)$ and let $q:\call(\mathcal{F}(X_E))\rightarrow  \call(\mathcal{F}(X_E))/\mathcal{K}$ be the quotient map. It is straightforward to check that $q(F)$ is a $C^*$-correspondence over $q(\mathcal{D})$ (with the obvious operations and inner product).
\begin{lem}\label{qF1}
	\begin{enumerate}
		\item[(1)] $\ker(\varphi_{q(F)})=\{0\}$
		\item[(2)] For $\xi,\eta \in F$ $$\theta_{q(\xi),q(\eta)}=\varphi_{q(F)}(q(\xi\eta^*)).$$
		\item[(3)] $K(q(F))=\call(q(F))=\varphi_{q(F)}(q(\cald)) $.
		
	\end{enumerate}
	
\end{lem}
\begin{proof}
	Note first that, for $d\in \cald$, we have (using (\ref{sumS})),
	\begin{equation}\label{d}
		d=dQ_0P_s^{\perp}+\Sigma_{|\alpha|=1} dS_{\alpha}S_{\alpha}^*+dP_s
	\end{equation}
	Assume now that $q(d)\in \ker(\varphi_{q(F)})$. Then, for $|\alpha|=1$, $q(dS_{\alpha})=0$ (as $S_{\alpha}\in F$) and, thus, $dS_{\alpha}\in \calk$ and also $dS_{\alpha}S_{\alpha}^*\in \calk$. We also showed that $Q_0$ and $P_s$ lie in $\calk$. It follows from (\ref{d}) that $q(d)=0$. This proves (1). In fact, for every $d\in \cald$ we get from (\ref{d}),
	$$q(d)=\Sigma_{|\alpha|=1}q(dS_{\alpha}S_{\alpha}^*).$$
	
	For (2) we compute $\theta_{q(\xi),q(\eta)}q(\zeta)=q(\xi)q(\eta)^*q(\zeta)=q(\xi\eta^*)q(\zeta)=\varphi_{q(F)}(q(\xi)q(\eta)^*)q(\zeta)$.
	This shows that $K(q(F))\subseteq \varphi_{q(F)}(q(\cald))$. For the converse, we have, for $d\in \cald$, $\varphi_{q(F)}(q(d))=\Sigma_{|\alpha|=1} \varphi_{q(F)}(q((dS_{\alpha})S_{\alpha}^*))=\Sigma_{|\alpha|=1} \theta_{q(dS_{\alpha}),q(S_{\alpha})} \in K(q(F))$.
\end{proof}

We now write, for $e\in E^1$ and $v\in E^0$, $u_e:=q(S_e)$, $q(P_v)=p_v$ and also $z:=q(Z)$. For a path $\alpha=\alpha_1 \alpha_2 \ldots \alpha_k$, we write $u_{\alpha}=u_{\alpha_1}u_{\alpha_2}\cdots u_{\alpha_k}$. So that $u_{\alpha}=q(S_{\alpha})$.

The following properties of $\{u_{\alpha}\}$ follow immediately from the properties of $\{S_{\alpha}\}$ and will be used repeatedly.

\begin{lem}\label{alpha}
	\begin{enumerate}
		\item [(1)] Let $\alpha,\beta$ are two paths of the same length.
		Then $u_{\alpha}^*u_{\alpha}=p_{s(\alpha)}$,  and $u_{\alpha}^*u_{\beta}=0$ whenever
		$\alpha\neq\beta$ .
		
		\item[(2)] For every $n>0$,$\Sigma_{|\alpha|=n}u_{\alpha}u_{\alpha}^*=I$.
		\item[(3)] For every $n>0$ and $v\in E^0$ that is not a source, $\Sigma_{|\alpha|=n, r(\alpha)=v}u_{\alpha}u_{\alpha}^*=p_v$.
		\item[(4)] If $m\geq 1$, $z^m$ commutes with $p_v$ for every vertex $v$. (Since $Z\in \varphi_{\infty}(A)'$).
	\end{enumerate}
\end{lem}

\begin{lem}\label{qF}
	$$q(F)=\Sigma_{|\alpha|=1} u_{\alpha}q(\cald).$$
	More generally, if we write $F^n$ for $span\{a_1\cdots a_n :\; a_i\in F\}$ then
	$$q(F^n)=\Sigma_{|\alpha|=n}u_{\alpha}q(\cald).$$
\end{lem}
\begin{proof} Fix $n\geq 1$.
	We have $\Sigma_{|\alpha|=n} u_{\alpha}u_{\alpha}^*=I$ and, therefore, $q(F^n)=\Sigma_{|\alpha|=n} u_{\alpha}u_{\alpha}^*q(F^n)=\Sigma_{|\alpha|=n} u_{\alpha} (u_{\alpha}^*q(F^n))\subseteq \Sigma_{\alpha} u_{\alpha}q(\cald)$. The converse inclusion is obvious.	
\end{proof}

\begin{defn} The algebra $\mathcal{T}(X_E,Z)/\mathcal{K}$ will be called the \emph{weighted Cuntz-Krieger algebra} and will be written $C^*(E,Z)$.
	\end{defn} 

The main result of this section is the following theorem.

\begin{thm}\label{isomorphism}
	The $C^*$-algebra $C^*(E,Z)=\mathcal{T}(X_E,Z)/\mathcal{K}$ is $*$-isomorphic to the Cuntz-Pimsner algebra $\mathcal{O}(q(F),q(\mathcal{D}))$.
\end{thm}

The arguments used in the proof of the theorem are similar to the ones used in the proof of \cite[Theorem 2.10]{HS} but we present them here for completeness.

Recall that, given a $C^*$-correspondence $X$ over a $C^*$-algebra $A$ and a $C^*$-algebra $B$. A $^*$-representation $\pi:\mathcal{O}(X,A)\rightarrow B$ is given by a pair of maps:
a contractive map $T:X\rightarrow B$ and a $^*$-homomorphism $\sigma: A\rightarrow B$ such that, for $x,y\in X$ and $a,b\in A$,
\begin{enumerate}
	\item[(1)] $T(\varphi_X(a)x b)=\sigma(a)T(x)\sigma(b)$
	\item[(2)] $T(x)^*T(y)=\sigma(\langle x,y \rangle)$
	\item[(3)] $\sigma^{(1)}(\varphi_X(a))=\sigma(a)$ for every $a\in J_X:=\varphi_X^{-1}(K(X))\cap (\ker(\varphi_X))^{\perp}$ where $\sigma^{(1)}:K(X)\rightarrow B$ is a $^*$-homomorphism defined by $\sigma^{(1)}(\theta_{x,y})=T(x)T(y)^*$.
	
\end{enumerate}

\begin{lem}\label{homomorphism}
	With $q(F)$ and $q(\mathcal{D})$ as above, define
	$$ \sigma: q(\mathcal{D}) \rightarrow \mathcal{T}/\mathcal{K}$$ by $\sigma(q(d))=q(d)$ and
	$$T: q(F) \rightarrow \mathcal{T}/\mathcal{K}$$ by $T(q(\xi))=q(\xi)$.
	
	Then $\sigma$ and $T$ satisfy conditions (1)-(3) above and, thus, define a $^*$-homomorphism $\pi$ of $\mathcal{O}(q(F),q(\mathcal{D}))$ onto $\mathcal{T}/\mathcal{K}$.
	
\end{lem}
\begin{proof}
	It is clear that $\sigma$ is a well defined $^*$-homomorphism and $T$ is a contraction. It is also clear that $T$ is a bimodule map (over $\sigma$). For (2), we write
	$$T(q(\xi))^*T(q(\eta))=q(\xi)^*q(\eta)=q(\xi^*\eta)=\sigma(\langle q(\xi),q(\eta)\rangle). $$
	For (3), note first that it follows from Lemma~\ref{qF1} that $J_{q(F)}=q(\cald)$ and, for $d\in \cald$,
	$$\varphi_{q(F)}(q(d))=\Sigma_{|\alpha|=1} \theta_{q(dS_{\alpha}),q(S_{\alpha})}.$$
	Thus $\sigma^{(1)}(\varphi_{q(F)}(q(d)))=\Sigma_{|\alpha|=1} T(q(dS_{\alpha}))T(q(S_{\alpha}))^*=\Sigma_{\alpha} q(dS_{\alpha})q(S_{\alpha})^*=q(d)=\sigma(q(d))$ proving (3). Since $\calt$ is generated by $\{I,S_{\alpha},Z : \alpha \in E^1\}$ and $I, Z \in \cald$ and $S_{\alpha}\in F$, $\calt / \calk$ is generated by the images of $\sigma$ and $T$. Thus, $\pi$ is onto $\calt /\calk$.
	
\end{proof}

\begin{proof} (of Theorem~\ref{isomorphism})
	In view of Lemma~\ref{homomorphism} all we need to prove is that the $*$-homomorphism $\pi$ is injective but this follows from Theorem 6.4 of \cite{Ka}, since $\sigma$ is injective and $\pi$ admits a gauge action. Indeed, the injectivity of $\sigma$
	is clear from its definition, and if $d\in \mathcal{D}$ and $S_{\alpha}\in F$ ($\alpha\in E^1$), then $\gamma_t(d)=d$ and $\gamma_t(S_{\alpha})=e^{it}S_{\alpha}$. Hence, $\gamma_t(\sigma(q(d)))=\gamma_t(q(d))=q(d)=\sigma(q(d))$ and $\gamma_t(T(q(S_{\alpha})))=e^{it}T(q(S_{\alpha}))$.
\end{proof}


\begin{rem}\label{gauge} The algebra $\mathcal{T}$ carries a gauge group action given by $\{\gamma_t\}$ (see the discussion following Lemma~\ref{T}). Since it leaves $\mathcal{K}$ invariant, we have a gauge group action on $C^*(E,Z)=\mathcal{T}/\mathcal{K}$. Note that the $^*$-isomorphism constructed in the proof of the theorem intertwines the action of the gauge group on $C^*(E,Z)$ and the gauge group action on the Cuntz-Pimsner algebra $O(q(F),q(\mathcal{D}))$ ) and, thus, when we later discuss gauge-invariant ideals of $C^*(E,Z)$, there is no ambiguity.
	\end{rem}

\section{The algebra $q(\cald)$ }\label{3}

In this section we shall study the algebra $q(\cald)$ and its ideals. We shall pay special attention to those ideals that are \emph{invariant} in the following sense.

\begin{defn}\label{idinv}
	A subspace $X\subseteq q(\mathcal{D})$ is said to be invariant if 
		$$u_{\alpha}^*Xu_{\beta} \subseteq X$$ for all $\alpha,\beta \in E^1$.
 Using Lemma~\ref{qF}, if $X=J$ is an ideal, this is equivalent to
		$$q(F)^*Jq(F)\subseteq J .$$	
\end{defn}

Note that the automorphism group $\gamma=\{\gamma_t\}$, on $\call(\mathcal{F}(X_E))$, as defined in the discussion following Lemma~\ref{T}, can be used to define bounded projections $\Phi_j$ on $\call(\mathcal{F}(X_E))$ by
$$\Phi_j(X)=\frac{1}{2\pi}\int_0^{2\pi}e^{-ijt}\gamma_t(X)dt .$$

In particular, for $A\in \mathcal{T}$, $A\in \mathcal{D}$ if and only if $\Phi_0(A)=A$.

\begin{lem}\label{qD}
	$$q(\cald)=C^*(\{u_{\alpha}z^mu_{\beta}^*, u_{\alpha}^*z^mu_{\beta}:\;|\alpha|=|\beta|,\; m\geq 0\}).$$
\end{lem}
\begin{proof}
	Fix $a\in q(\cald)$ and write $a=q(A)$ for some $A\in \cald$. Since $A\in  \calt$, it is a norm limit of (noncommutative) polynomials in $Z, S_i, S_i^*$. Since $\Phi_0$ is a contractive projection, and $A=\Phi_0(A)$, we can assume, by applying $\Phi_0$, that each of these polynomials lie in $\cald$. Each such polynomial is a finite sum of monomials and each monomial lies in the range of some $\Phi_j$. By applying $\Phi_0$, we can assume that each monomial lies in $\cald$. Thus, $A$ is a norm limit of polynomials in $\{Z,S_i, S_i^*\}$ where each monomial lies in $\cald$. It follows that $a$ is a norm limit of polynomials in $\{z,u_i, u_i^*\}$ where each monomial lies in $q(\mathcal{D})$. It is straightforward to check (using Lemma~\ref{alpha}) that each monomial in $q(\cald)$ is a product of elements in $\{u_{\alpha}z^mu_{\beta}^*, u_{\alpha}^*z^mu_{\beta}:\;|\alpha|=|\beta|\geq 0,\; m\geq 0\}$.	
	
\end{proof}

Recall that a unital $C^*$-algebra is said to be \emph{finite} if it has no proper isometries.

\begin{lem}\label{finite}
	The algebra $q(\mathcal{D})$ is finite.
\end{lem}
\begin{proof}
	Suppose $v=q(V)$ is an isometry in $q(\mathcal{D})$. Then $v^*v=I$ and, thus $V^*V=I+K$ for a compact operator $K$. Thus $\lim_{k\rightarrow \infty}\|Q_k(V^*V-I)Q_k\|=0$. Fix $N$ such that, for every $k\geq N$, $\|Q_kV^*VQ_k-Q_k\|<1/2$. It follows that, for such $k$, $Q_kV^*VQ_k$ is invertible in $\call(X_E^{\otimes k})$. Write $C_k$ for its inverse and note that
	\begin{equation}\label{Cbdd} 
		\|C_k\|\leq \frac{1}{1-\|Q_k-Q_kV^*VQ_k\|}<2.
	\end{equation}  
	Since $X_E^{\otimes k}$ is a finite dimensional space and $Q_kV^*VQ_k=Q_kV^*Q_kVQ_k$ is invertible in $\call(X_E^{\otimes k})$, it follows that $Q_kVQ_k$ is also invertible there. Write $B_k$ for its inverse (so that $C_k=B_k^*B_k$) and $B:=I\oplus I \oplus \cdots \oplus I \oplus B_N\oplus B_{N+1} \oplus \cdots \in \call(\calf(X_E))$. It follows from (\ref{Cbdd}) that $B$ is bounded (in fact $\|B\|\leq \sqrt{2}$) and it is the inverse of $V':=I\oplus I \oplus \cdots \oplus I \oplus Q_NVQ_N\oplus Q_{N+1}VQ_{N+1} \oplus \cdots$. Thus $V'$ is invertible in $\call(\calf(X_E))$ and $q(V')$ is invertible in $\call(\calf(X_E))/\mathcal{K}$. But $q(V')=q(V)=v$. Thus $v$ is invertible in $\call(\calf(X_E))/\mathcal{K}$. Since $v\in q(\mathcal{D})$, it is also invertible there.	
	
\end{proof}

We shall say that the graph is a \emph{cycle with an entry} and write $E=C_{k,l}$ for $1\leq l \leq k$ if $E^0=\{v_1,\ldots ,v_l,\ldots, v_k\}$ and $E^1=\{e_1,\ldots,e_{k}\}$ with $e_i: v_i\rightarrow v_{i+1}$ if $1\leq i \leq k-1$ and $e_{k}:v_{k}\rightarrow v_l$. (If $l=1$ we get a cycle and we write $C_k$ for $C_{k,1}$).

\begin{lem}\label{cycleentry}
	For an irreducible, finite graph with no sinks, the following conditions are equivalent.
	\begin{enumerate}
		\item [(1)] $E$ is a cycle with an entry.
		\item[(2)] Given $n\in \mathbb{N}$ and $v\in E$, there is a unique path $\alpha$ with $|\alpha|=n$ and $s(\alpha)=v$.
		\item[(3)] For every $v\in E^0$, $|s^{-1}(v)|=1$.
		\item[(4)] The algebra $q(\cald)$ is commutative.
	\end{enumerate}
\end{lem}
\begin{proof}
	It is clear that (1) implies (2). Also, (2) and (3) are equivalent (in fact, (3) is (2) with $n=1$ and (2) follows from (3) by successive applications of (3)).
	
	Suppose (2) holds. Then  $|s^{-1}(v)|=1$ for all $v$ and thus $|E^0|=|E^1|$ and 
	\begin{equation}\label{sr}|E^1|=\Sigma_v |s^{-1}(v)|=\Sigma_v|r^{-1}(v)|.
	\end{equation} To prove (1), assume first that $E$ has no sources. Therefore $|r^{-1}(v)|\geq 1$ for every $v$. But then, from (\ref{sr}), we get $|r^{-1}(v)|=1$ for all $v$. Since $E$ is finite and irreducible, it follows that it is a cycle.
	Assume now that $E$ has sources. Suppose there are two sources, $v_1,w$. Since $E$ is irreducible, there is either a path $\alpha$ with $s(\alpha)=v_1$ and $r(\alpha)=w$ or a path $\beta$ with $s(\beta)=w$ and  $r(\beta)=v_1$. But this is impossible since both $v_1$ and $w$ are assumed to be sources. Thus, there is only one source, say $v_1$. Since 
	$|r^{-1}(v_1)|=0$ and $|r^{-1}(v)|\geq 1$ for all other $v$, it follows from (\ref{sr})   that there is one vertex, say $v_l$ with $|r^{-1}(v_l)|=2$ and, for every $v$ different from $v_1$ and $v_l$, $|r^{-1}(v)|=1$. It is easy to check that, in this case, $E=C_{k,l}$.
	
	Now assume that (2) holds and prove (4). Each summand of the Fock correspondence, viewed simply as a vector space, has a finite basis. For the zero'th term it is $\{\delta_v: v\in E^0\}$ and for the $k$'th term it is $\{\delta_{\alpha} : |\alpha|=k \}$. Putting these together, we get a countable basis for the Fock correspondence (viewed as a vector space). Once we show that each generator of $\mathcal{D}$ is diagonal with respect to this basis, we are done.

	First note that $Z$ is diagonal. For that, fix an element $\delta_{\gamma}$ of the basis and let $k=|\gamma|$. Then $Z\delta_{\gamma}=Z_k \delta_{\gamma}$. Since $Z_k$ is a bimodule map, $Z \delta_{\gamma}$ will be a linear combination of $\{\delta_{\alpha}: |\alpha|=k, s(\alpha)=s(\gamma), r(\alpha)=r(\gamma)\}$. But, since $E$ satisfies (2), the latter set is $\{\delta_{\gamma}\}$ showing that $Z$ is diagonal.
	
	Now consider $S_{\alpha}Z^mS_{\beta}^*\delta_{\gamma}$ (with $|\alpha|=|\beta|$). This is $0$ unless $\gamma=\beta \gamma'$ (and then $s(\beta)=r(\gamma')$). In this case, we get $S_{\alpha}Z^m\delta_{\gamma'}$ and, since $Z^m$ was shown to be diagonal, this lies in $\mathbb{C}S_{\alpha}\delta_{\gamma'}=\mathbb{C}\delta_{\alpha \gamma'}$. But this is $0$ unless $s(\alpha)=r(\gamma')=s(\beta)$. Since $E$ satisfies (2), it follows that $\alpha=\beta$ (as $s(\alpha)=s(\beta)$ and $|\alpha|=|\beta|$). This shows that $S_{\alpha}Z^mS_{\beta}^*$ is diagonal. The proof for $S_{\alpha}^*Z^mS_{\beta}$ is similar and is omitted. This proves (4).
	
	Conversely, suppose (3) does not hold. Then there is a vertex $v$ with $|s^{-1}(v)|>1$. So we can fix $e_1,e_2\in E^1$ such that $s(e_1)=s(e_2)$ but $e_1\neq e_2$. Write $u_i$ for $u_{e_i}$ and consider $u_1u_1^*, u_1u_2^* \in q(\cald)$. We have $u_1u_1^*u_1u_2^*=u_1u_2^*\neq 0$ but $u_1u_2^*u_1u_1^*=0$ (as $u_2^*u_1=0$). Thus (4) does not hold and it shows that (4) implies (3).	
	
\end{proof}


\begin{lem}\label{cycle}
	Let $E=C_{k,l}$. Then for $m\geq 0$ and two different paths of the same length  $\alpha,\beta$, we have $u_{\alpha}^*z^mu_{\beta}=u_{\alpha}z^mu_{\beta}^*=0$.
	
	Also, if $E$ is a cycle, for every path $\alpha$ we have $u_{\alpha}u_{\alpha}^*=p_{r(\alpha)}$.
\end{lem}
\begin{proof}

	Note that $u_{\beta}u_{\beta}^*$ lies in $q(\mathcal{D})$ and, by Lemma~\ref{cycleentry}(4), commutes with $z^m$. Thus $u_{\alpha}^*z^mu_{\beta}=u_{\alpha}^*z^mu_{\beta}u_{\beta}^*u_{\beta}=u_{\alpha}^*u_{\beta}u_{\beta}^*z^mu_{\beta}=0$ (as $u_{\alpha}^*u_{\beta}=0$).	Also, if $u_{\alpha}u_{\beta}^*\neq 0$ then $s(\alpha)=s(\beta)$. But, since $|\alpha|=|\beta|$, they cannot be different paths (by Lemma~\ref{cycleentry}(2)). Thus $u_{\alpha}u_{\beta}^*=0$ and $u_{\alpha}z^mu_{\beta}^*=u_{\alpha}u_{\alpha}^*u_{\alpha}z^mu_{\beta}^*=u_{\alpha}z^mu_{\alpha}^*u_{\alpha}u_{\beta}^*=0$.
	
	The last statement follows from Lemma~\ref{alpha} (3) since there is only one path of a given length that ends in $v$.
\end{proof}

For the next proposition note that we can view $q(\mathcal{D})$ as a $C^*$-correspondence over itself with left and right actions given by multiplication and the inner product is $\langle d_1,d_2\rangle:=d_1^*d_2$.

\begin{prop}\label{nonperiodic}
	Assume that the graph $E$ has no sinks (so that $|s^{-1}(v)|>0$ for every $v\in E^0$).
	\begin{enumerate}
		\item[(1)] Suppose $E$ is not a cycle with entry, then there is no map $w:q(\cald)\rightarrow q(F^n)$ that is an isomorphism of correspondences (that is, a surjective bimodule map that preserves the inner product).
		\item[(2)] Suppose $E=C_{k,l}$. Then every $u_{\alpha}$ for $|\alpha|=n$ commutes with $q(\cald)$ if and only if  there is a map $w:q(\cald)\rightarrow q(F^n)$ that is an isomorphism of correspondences .
	\end{enumerate} 
\end{prop}
\begin{proof}
	Assume first that $E$ is not a cycle with entry. That implies that there is a vertex, say $v_0$, such that $|s^{-1}(v_0)|>1$. Fix $n\geq 1$ and $w:q(\cald)\rightarrow q(F^n)$ that is an isomorphism of correspondences. Since there are no sinks, we can find, for every $v\in E^0$, a path $\alpha_v$ whose length is $n$ and $s(\alpha_v)=v$. Since $|s^{-1}(v_0)|>1$, we can find two different paths $\alpha_{v_0}$ and $\alpha_{v_0}'$ with $|\alpha_{v_0}|=|\alpha_{v_0}'|=n$ and $s(\alpha_{v_0})=s(\alpha_{v_0}')=v_0$.
	For every pair $\alpha,\beta$ of different paths of length $n$,  $u_{\alpha},u_{\beta}\in q(F^n)$ and, thus, there are $d_{\alpha},d_{\beta}\in q(\mathcal{D})$ such that $w(d_{\alpha})=u_{\alpha}$ and $w(d_{\beta})=u_{\beta}$. Since $w$ preserves inner products, we have $d_{\alpha}^*d_{\beta}=\langle d_{\alpha},d_{\beta}\rangle=\langle w(d_{\alpha}),w(d_{\beta})\rangle = \langle u_{\alpha},u_{\beta}\rangle=u_{\alpha}^*u_{\beta}=0$. Similarly $d_{\alpha}^*d_{\alpha}=p_{s(\alpha)}$. 
	Now write $d=\Sigma_{v\in E^0}d_{\alpha_v}$ and note that it is an isometry in $q(\mathcal{D})$ but its range is orthogonal to the range of $d_{\alpha_{v_0}'}$. By Lemma~\ref{finite}, this is impossible, proving (1).
	
	Now, assume that $E$ is  $C_{k,l}$ and $w:q(\cald)\rightarrow q(F^n)$ is an isomorphism of correspondences. For every $v\in E^0$, let $\alpha_v$ be the (unique) path of length $n$ that starts at $v$ and write $u_v=u_{\alpha_v}$ so that, by Lemma~\ref{qF}, $F^n=\Sigma_v u_v q(\mathcal{D})$. Fix $a\in q(\cald)$. For every $v\in E^0$ there is $c_v\in q(\cald)$ such that $w(c_v)=u_{v}$ and $\{b_v\}$ in $q(\cald)$ such that $w(a)=\Sigma_v u_v b_v$. Thus $w(a)=\Sigma_v w(c_v)b_v=w(\Sigma_v c_vb_v)$ . It follows that $a=\Sigma c_vb_v$.  Now, $u^{*}_vau_v=\langle w(c_v),w(ac_v) \rangle=\langle c_v,ac_v \rangle =c_v^*ac_v=c_v^*c_va$ (since $q(\cald)$ is commutative in this case). Similarly, for $a=I$, $u_v^*u_v=c_v^*c_v$ and we get $u_v^*au_v=u_v^*u_va$. Since $u_v$ is a partial isometry and $u_vu_v^*$ commutes with $a$ (as elements of $q(\mathcal{D})$), we get $au_v=u_va$. 
	
	For the other direction, assume that every $a\in q(\cald)$ commutes with each $u_{\gamma}$ for a path $\gamma$ of lenth $n$.

	We define $w:q(\cald) \rightarrow q(F^n)$ by $$w(a)=\Sigma_{|\gamma|=n} u_{\gamma}a=\Sigma_{|\gamma|=n}au_{\gamma}. $$ Clearly $w$ is a linear bimodule map. By Lemma~\ref{qF} $w$ is surjective. It is left to show that it preserves the inner products and, for that, we compute, for $a,b\in q(\cald)$,
	$$\langle w(a),w(b)\rangle=w(a)^*w(b)=\Sigma_{|\gamma|=|\gamma'|=n} a^*u_{\gamma}^*u_{\gamma'}b=\Sigma_{|\gamma|=n} a^*p_{s(\gamma)}b=a^*b=\langle a,b\rangle.$$

\end{proof}

For a fixed $p\in \mathbb{N}$ we shall consider the following condition on $Z$:
\vspace{5mm}

\textbf{Condition A(p):} $\lim_{k\rightarrow \infty} \|I_p\otimes Z_k-Z_{k+p}\|=0$.

\vspace{5mm}
In particular, Condition A(p) holds if there is some $N>0$ such that, for $k\geq N$, $Z_{k+p}=I_{p}\otimes Z_k$. Note that, if $\alpha$ is a path of length $p+k$ and $\alpha=\beta \gamma$ is the unique way of writing it with $|\beta|=p$ and $|\gamma|=k$, then $$(I_p\otimes Z_k)\delta_{\alpha}=\delta_{\beta}\otimes Z_k\delta_{\gamma}.$$

\begin{lem}\label{uz}
	$Z$ satisfies Condition A(p) if and only if $u_{\alpha}z=zu_{\alpha}$ for every path $\alpha$ with $|\alpha|=p$. In this case, we also have $u_{\alpha}^*z=zu_{\alpha}^*$ for $|\alpha|=p$.
\end{lem}
\begin{proof}
	$u_{\alpha}z=zu_{\alpha}$ if and only if $S_{\alpha}Z-ZS_{\alpha}\in \calk$. This hold if and only if $S_{\alpha}ZS_{\alpha}^*-ZS_{\alpha}S_{\alpha}^*\in \calk$ (since $S_{\alpha}^*S_{\alpha}=P_{s(\alpha)}$ and $Z$ commutes with it). Thus, $u_{\alpha}z=zu_{\alpha}$ for every $|\alpha|=p$ if and only if $\Sigma_{|\alpha|=p} S_{\alpha}ZS_{\alpha}^*-\Sigma_{|\alpha|=p}ZS_{\alpha}S_{\alpha}^*\in \calk$. Since $I-\Sigma_{|\alpha|=p} S_{\alpha}S_{\alpha}^*=\Sigma_{i=0}^{p-1}Q_iP_s^{\perp}+P_s\in \calk$, this is satisfied if and only if $\Sigma_{|\alpha|=p} S_{\alpha}ZS_{\alpha}^*-Z\in \calk$. The last inclusion is equivalent to $||Q_{k+p}(\Sigma_{|\alpha|=p} S_{\alpha}ZS_{\alpha}^*)Q_{k+p}-Q_{k+p}ZQ_{k+p}||\rightarrow 0$ for $k\rightarrow \infty$. 
	But $Q_{k+p}ZQ_{k+p}=Z_{k+p}$ and it is easy to check that $Q_{k+p}(\Sigma_{|\alpha|=p} S_{\alpha}ZS_{\alpha}^*)Q_{k+p}=I_p\otimes Z_k$ (simply apply it to $\delta_{\gamma \beta}$ for $|\gamma|=p$ and $\beta|=k$). So this completes the proof.

\end{proof}

From now on, we assume that the weight sequence $Z$ satisfies Condition A(p) for some $p\geq 1$. We also write $q$ for $p-1$.

\begin{lem}\label{zcomm}
	Assume Condition A(p) and let $c$ be in $C^*(z)$. Then
	\begin{enumerate}
		\item[(1)] For paths $\alpha,\beta$ with $|\alpha|=|\beta|$, $cu_{\alpha}u_{\beta}^*=u_{\alpha}u_{\beta}^*c$
		\item [(2)] For paths $\alpha$,$\beta$, we have $u_{\alpha}^*cu_{\beta}=u_{\alpha}^*cu_{\alpha}u_{\alpha}^*u_{\beta}$.
		\item[(3)] If $|\alpha|=|\beta|$ but $\alpha\ne \beta$ then $u_{\alpha}^*cu_{\beta}=0$.
		\item[(4)] For paths $\mu$ and $\xi$ with $s(\xi)=r(\mu)$ , $u_{\mu}^*u_{\xi}^*cu_{\xi}=u_{\mu}^*u_{\xi}^*cu_{\xi}u_{\mu}u_{\mu}^*$.
	\end{enumerate}	
\end{lem}
\begin{proof} Suppose $|\alpha|=|\beta|=m$ and fix $n$ such that $m\leq np$.
	Since $\Sigma_{|\gamma|=np-m}u_{\gamma}u_{\gamma}^*=I$ and  $z$ is commutes with $u_{\delta}$ for $|\delta|=np$, we have
	$$u_{\alpha}u_{\beta}^*z=\Sigma_{|\gamma|=np-m} u_{\alpha}u_{\gamma}u_{\gamma}^*u_{\beta}^*z=\Sigma_{|\gamma|=np-m} u_{\alpha}u_{\gamma}zu_{\gamma}^*u_{\beta}^*=\Sigma_{|\gamma|=np-m} zu_{\alpha}u_{\gamma}u_{\gamma}^*u_{\beta}^*=zu_{\alpha}u_{\beta}^*. $$
	(Note that, if $s(\alpha)\neq r(\gamma)$, $u_{\alpha}u_{\gamma}=0$ but the equalities still hold). This proves (1).
	
	For (2) we use (1) and the fact that $u_{\alpha}$ is a partial isometry to compute $u_{\alpha}^*cu_{\beta}=u_{\alpha}^*u_{\alpha}u_{\alpha}^*cu_{\beta}=u_{\alpha}^*cu_{\alpha}u_{\alpha}^*u_{\beta}$.
	
	Part (3) follows from (2) and Lemma~\ref{alpha} (1).
	
	For (4), we use (2) with $\alpha=\xi \mu$ and $\beta=\xi$ ( recall that $u_{\xi}^*u_{\xi}=p_{s(\xi)}=p_{r(\mu)}$).

\end{proof}

\begin{lem}\label{qDAp}
	If, for some $p>0$, condition A(p) holds then 
	$$q(\cald)=C^*(\{u_{\alpha}z^mu_{\beta}^*:\;|\alpha|=|\beta|,\; m\geq 0\})$$ $$=C^*(\{u_{\alpha}z^mu_{\beta}^*:\;|\alpha|=|\beta|,\; m\geq 0, \; s(\alpha)=s(\beta)\}).$$	
\end{lem}
\begin{proof}
	We need to show that, for $|\alpha|=|\beta|$ and $m\geq 0$, $u_{\alpha}^*z^mu_{\beta}\in C^*(\{u_{\alpha}z^mu_{\beta}^*:\;|\alpha|=|\beta|,\; m\geq 0\})$. 
	First, write $\alpha=\alpha'\alpha''$ and $\beta=\beta'\beta''$ where $|\alpha'|=|\beta'|$ is a multiple of $p$ and $0\leq |\alpha''|=|\beta''|=k<p$. Then $u_{\alpha}^*z^mu_{\beta}=u_{\alpha''}^*u_{\alpha'}^*z^mu_{\beta'}u_{\beta''}=u_{\alpha''}^*z^mu_{\alpha'}^*u_{\beta'}u_{\beta''}$ which is $0$ if $\alpha'$ is different from $\beta'$ and $u_{\alpha''}^*z^mu_{\beta''}$ otherwise. Thus, we can assume that $|\alpha|=|\beta|=k<p$.
	Then, we write
	$$u_{\alpha}^*z^mu_{\beta}=\Sigma_{|\gamma|=p-k} u_{\gamma}u_{\gamma}^*u_{\alpha}^*z^mu_{\beta}=
	\Sigma_{|\gamma|=p-k} u_{\gamma}(u_{\alpha}u_{\gamma})^*z^mu_{\beta}=\Sigma_{|\gamma|=p-k} u_{\gamma}z^m(u_{\alpha}u_{\gamma})^*u_{\beta}$$ which is $0$ if $\alpha$ is different from $\beta$ and $\Sigma u_{\gamma}z^mu_{\gamma}^*$ otherwise. In any case, we see that $u_{\alpha}^*z^mu_{\beta}$ lies in 
	$C^*(\{u_{\alpha}z^mu_{\beta}^*:\;|\alpha|=|\beta|,\; m\geq 0\})$.
	
	Finally note that, by Lemma~\ref{alpha}(4), $z^m$ commutes with $p_{s(\alpha)}$ and, thus, $u_{\alpha}z^mu_{\beta}^*=0$ unless $s(\alpha)=s(\beta)$.
\end{proof}

We now write for $n\geq 0$,
\begin{equation}\label{Cn}
	\mathcal{C}_n=C^*(\{u_{\alpha}z^lu_{\beta}^* :\; l\geq 0,\; np\leq |\alpha|=|\beta|<(n+1)p,\; s(\alpha)=s(\beta)\}).
\end{equation} 
In particular,
\begin{equation}\label{C0}
	\mathcal{C}_0=C^*(\{u_{\alpha}z^lu_{\beta}^* :\; l\geq 0,\; 0\leq |\alpha|=|\beta|<p,\; s(\alpha)=s(\beta)\}).
\end{equation}

\begin{lem}\label{inclusion}
	Assuming Condition A(p) for some $p>0$, 
	for every $n\geq 0$, $\mathcal{C}_n\subseteq \mathcal{C}_{n+1}$ . If $\iota_n$ is the inclusion map then $$\iota_n(u_{\alpha}z^lu_{\beta}^*)=\Sigma_{|\gamma|=p} u_{\alpha}u_{\gamma}z^lu_{\gamma}^*u_{\beta}^*.$$
	
	
\end{lem}
\begin{proof}
	We take $u_{\alpha}z^lu_{\beta}^*\in \mathcal{C}_n$ (with $np\leq |\alpha|=|\beta|	<(n+1)p$) and compute
	$$u_{\alpha}z^lu_{\beta}^*=\Sigma_{|\gamma|=p} u_{\alpha}u_{\gamma}u_{\gamma}^*z^lu_{\beta}^*=\Sigma_{|\gamma|=p} u_{\alpha}u_{\gamma}z^lu_{\gamma}^*u_{\beta}^*\in \mathcal{C}_{n+1}$$ since $(n+1)p\leq|\alpha \gamma|=|\beta \gamma|< (n+2)p$.
\end{proof}

Thus
$$q(\cald)=\overline{\cup_n \mathcal{C}_n}.$$

\begin{lem}\label{eCnf}
	\begin{enumerate}
		\item[(1)]	For every $n\geq 0$ and every $e,f\in E^1$, $$u_e^*\calc_n u_f\subseteq \calc_n$$ so that $\calc_n$ is invariant (see Definition~\ref{idinv}).
		\item[(2)] If $m<n$, $$\calc_m=\bigvee_{|\gamma_i|=(n-m)p} u_{\gamma_1}^*\calc_n u_{\gamma_2}.$$
		\item[(3)] If $n\geq 0$, Condition A(p) holds, $np\leq|\alpha_1|=|\alpha_2|=|\alpha_3|=|\alpha_4|<(n+1)p$ and $|\gamma_1|=|\gamma_2|$ then
		$$u_{\alpha_1}u_{\gamma_1}^*u_{\alpha_2}^*\calc_n u_{\alpha_3}u_{\gamma_2}u_{\alpha_4}^*\subseteq \calc_n.$$
	\end{enumerate}
\end{lem}
\begin{proof}
	Note that, if $c_1, c_2\in \calc_n$ and $e,f\in E^1$ then 
	$$u_e^*c_1c_2u_f=\sum_{g\in E^1} u_e^*c_1u_gu_g^*c_2u_f$$ and, thus, to prove (1), it suffices to show that, for every generator $u_{\alpha}z^lu_{\beta}^*$ of $\calc_n$ (where $l\geq 0$, $np\leq|\alpha|=|\beta|<(n+1)p$ and $s(\alpha)=s(\beta)$), we have $$u_e^*u_{\alpha}z^lu_{\beta}^*u_f\in \calc_n .$$
	Assume first that $|\alpha|=|\beta|>np$. Then we can write $\alpha=h\alpha'$ and $\beta=g\beta'$ for some $h,g\in E^1$ and $np\leq |\alpha'|=|\beta'|< (n+1)p$ and we have $u_e^*u_{\alpha}z^lu_{\beta}^*u_f=\delta_{e,h}\delta_{f,g}u_{\alpha'}z^lu_{\beta'}^*\in \mathcal{C}_n$. Now assume that $|\alpha|=|\beta|=np$ and $n\neq 0$. Then, a similar argument shows that $u_e^*u_{\alpha}z^lu_{\beta}^*u_f=\delta_{e,h}\delta_{f,g}u_{\alpha'}z^lu_{\beta'}^*$ and $|\alpha'|=|\beta'|=np-1<np$. Thus, in this case, $u_e^*u_{\alpha}z^lu_{\beta}^*u_f\in \mathcal{C}_{n-1}\subseteq \mathcal{C}_n$ (where we used Lemma~\ref{inclusion}). To complete the proof we have to show that $u_e^*z^lu_f\in \mathcal{C}_0$. For that, we write $u_e^*z^lu_f=\Sigma_{|\gamma|=p-1}u_{\gamma}u_{\gamma}^*u_e^*z^lu_f=\Sigma_{|\gamma|=p-1}u_{\gamma}z^lu_{\gamma}^*u_e^*u_f=\delta_{e,f}\Sigma_{|\gamma|=p-1}u_{\gamma}z^lu_{\gamma}^*\in \mathcal{C}_0$. This proves part (1).
	
	For part (2), fix $\gamma_1,\gamma_2$ with $|\gamma_i|=(n-m)p$ and let $u_{\alpha_1}z^lu_{\alpha_2}^*$ be an arbitrary generator of $\calc_n$ (so that $l\geq 0$, $np\leq |\alpha_i|<(n+1)p$ and $s(\alpha_1)=s(\alpha_2)$). Then consider $b:=u_{\gamma_1}^*u_{\alpha_1}z^lu_{\alpha_2}^*u_{\gamma_2}$. If $b\neq 0$ then both $u_{\gamma_1}^*u_{\alpha_1}$ and $u_{\gamma_2}^*u_{\alpha_2}$ are non zero which implies that we can write $\alpha_i=\gamma_i \delta_i$ with $|\delta_i|=|\alpha_i|-(n-m)p$. In this case, $b=u_{\delta_1}z^lu_{\delta_2}^*\in \calc_m$. Since this holds for every generator  of $\calc_n$,  and $\calc_m$ is a $C^*$-algebra, we get $$\bigvee_{|\gamma_i|=(n-m)p} u_{\gamma_1}^*\calc_n u_{\gamma_2}\subseteq \calc_m .$$ For the other direction, fix a generator $u_{\mu_1}z^lu_{\mu_2}^*$ of $\calc_m$ (so that $mp\leq |\mu_i|<(m+1)p\leq np$ and $s(\mu_1)=s(\mu_2)$). Then, for $\gamma_i$ with $|\gamma_i|=(n-m)p$ and $s(\gamma_i)=r(\mu_i)$,
	$$u_{\mu_1}z^lu_{\mu_2}^*= u_{\gamma_1}^*u_{\gamma_1}u_{\mu_1}z^lu_{\mu_2}^*u_{\gamma_2}^*u_{\gamma_2}\in \bigvee_{|\gamma_i|=(n-m)p} u_{\gamma_1}^*\calc_n u_{\gamma_2} $$ since $u_{\gamma_1}u_{\mu_1}z^lu_{\mu_2}^*u_{\gamma_2}^*\in \calc_n$. (Note that such $\gamma_i$ exist since $E$ has no sinks.) Since this holds for every generator  of $\calc_m$ and the set on the right-hand side is a $C^*$-algebra, this proves (2).

	For part (3), note that, since $|\alpha_2 \gamma_1|\geq |\alpha_1|$, we can write $\alpha_2 \gamma_1=\mu_1 \mu_2$ where $|\mu_2|=|\alpha_2|$ and, similarly, $\alpha_3\gamma_2=\nu_1\nu_2$ with $|\nu_2|=|\alpha_3	|$. Thus $u_{\alpha_1}u_{\gamma_1}^*u_{\alpha_2}^*\calc_n u_{\alpha_3}u_{\gamma_2}u_{\alpha_4}^*=u_{\alpha_1}u_{\mu_2}^*u_{\mu_1}^*\calc_n u_{\nu_1}u_{\nu_2}u_{\alpha_4}^*$. But now the result follows because $u_{\alpha_1}u_{\mu_2}^*$ and $u_{\nu_2}u_{\alpha_4}^*$ are in $\calc_n$ by definition and $u_{\mu_1}^*\calc_n u_{\nu_1}$ is contained in $\calc_n$ by successively applying part (1).
	
\end{proof}

\begin{lem}\label{Lempsi}
	\begin{enumerate}
		\item[(1)]	For $c\in \mathcal{C}_0$, $|\theta_1|=|\theta_2|=q$ and $|\mu_1|=|\mu_2|$ with $\mu_1\neq \mu_2$,
		\begin{equation}
			u_{\mu_1}^*u_{\theta_1}^*cu_{\theta_2}u_{\mu_2}=0.
		\end{equation}
		\item[(2)] For $c\in \mathcal{C}_0$ and $|\theta_1|=|\theta_2|=q$,
		$$u_{\theta_1}^*cu_{\theta_2}\neq 0$$ only if $s(\theta_1)=s(\theta_2)$.
		\item[(3)] For $a\in \calc_n$ and $|\gamma_1|=|\gamma_2|=np+q$,
		$$u_{\gamma_1}^*au_{\gamma_2}\neq 0$$ only if $s(\gamma_1)=s(\gamma_2)$.
	\end{enumerate} 	
\end{lem}
\begin{proof}
	From the definition of $\mathcal{C}_0$, we can assume that 
	$$c=u_{\delta_1}z^{l_1}u_{\nu_1}^*u_{\delta_2}z^{l_2}u_{\nu_2}^* \cdots u_{\delta_n}z^{l_n}u_{\nu_n}^* $$  for $0\leq |\delta_i|=|\nu_i|=k_i\leq q$ and $l_i\geq 0$ for every $1\leq i \leq n$.
	
	Recall that $\Sigma_{|\gamma|=q}u_{\gamma}u_{\gamma}^*=I$ and, thus, $	u_{\mu_1}^*u_{\theta_1}^*cu_{\theta_2}u_{\mu_2}$ can be written as a sum of elements of the form
	$$a=u_{\mu_1}^*u_{\theta_1}^*u_{\delta_1}z^{l_1}u_{\nu_1}^*u_{\gamma_1}u_{\gamma_1}^*u_{\delta_2}z^{l_2}u_{\nu_2}^*u_{\gamma_2}u_{\gamma_2}^*u_{\delta_3} \cdots u_{\delta_n}z^{l_n}u_{\nu_n}^*u_{\theta_2}u_{\mu_2}$$ with $|\gamma_i|=q$. So, to prove (1), it suffices to show that (assuming $\mu_1\neq \mu_2$) each such term is $0$. Thus, we now fix $a$ as above.
	
	Since $|\delta_1|\leq |\theta_1|=q$, we can write $\theta_1=\theta_1'\xi_1$ with $|\theta_1'|=|\delta_1|$ and $|\xi_1|=q-|\delta_1|=q-k_1$ so that $u_{\theta_1}^*u_{\delta_1}=u_{\xi_1}^*u_{\theta_1'}^*u_{\delta_1}$. But the last term is non zero only if $\theta_1'=\delta_1$ and, in this case, $u_{\theta_1}^*u_{\delta_1}=u_{\xi_1}^*$. Thus we can assume that $u_{\theta_1}^*u_{\delta_1}=u_{\xi_1}^*$ with $|\xi_1|=q-k_1$. Similarly we assume that $u_{\nu_1}^*u_{\gamma_1}=u_{\eta_1}$ with $|\eta_1|=q-k_1$, $u_{\gamma_1}^*u_{\delta_2}=u_{\xi_2}^*$ ($|\xi_2|=q-k_2$) and so on and finally $u_{\nu_n}^*u_{\theta_2}=u_{\eta_n}$ ($|\eta_n|=q-k_n$).
	Thus, we can assume that 
	$$a=u_{\mu_1}^*u_{\xi_1}^*z^{l_1}u_{\eta_1}u_{\xi_2}^*z^{l_2}u_{\eta_2} \cdots u_{\xi_n}^*z^{l_n}u_{\eta_n}u_{\mu_2} $$ with $|\xi_i|=|\eta_i|$ for all $i$. But, using Lemma~\ref{zcomm}(3), we can assume that $\xi_i=\eta_i$ for all $i$. Thus
	$$a=u_{\mu_1}^*u_{\xi_1}^*z^{l_1}u_{\xi_1}u_{\xi_2}^*z^{l_2}u_{\xi_2} \cdots u_{\xi_n}^*z^{l_n}u_{\xi_n}u_{\mu_2}. $$ Now applying Lemma~\ref{zcomm}(4) successively , we get
	$$a=u_{\mu_1}^*u_{\xi_1}^*z^{l_1}u_{\xi_1}u_{\mu_1}u_{\mu_1}^*u_{\xi_2}^*z^{l_2}u_{\xi_2} \cdots u_{\xi_n}^*z^{l_n}u_{\xi_n}u_{\mu_2}=\ldots =$$ $$=u_{\mu_1}^*u_{\xi_1}^*z^{l_1}u_{\xi_1}u_{\mu_1}u_{\mu_1}^*u_{\xi_2}^*z^{l_2}u_{\xi_2}u_{\mu_1}u_{\mu_1}^* \cdots u_{\xi_n}^*z^{l_n}u_{\xi_n}u_{\mu_1}u_{\mu_1}^*u_{\mu_2}$$ and this is $0$ since $|\mu_1|=|\mu_2|$ but $\mu_1\neq \mu_2$. This proves (1). Part (2) follows from (1) with $|\mu_1|=|\mu_2|=0$.
	
	To prove part (3) fix $a\in \calc_n$ and paths $\gamma_1,\gamma_2$ of length $np+q$ and $s(\gamma_1)\neq s(\gamma_2)$. For every $\mu_1,\mu_2$ of length $np$, write $a(\mu_1,\mu_2)=u_{\mu_1}u_{\mu_1}^*au_{\mu_2}u_{\mu_2}^*$ and consider $u_{\gamma_1}^*a(\mu_1,\mu_2)u_{\gamma_2}=(u_{\gamma_1}^*u_{\mu_1})(u_{\mu_1}^*au_{\mu_2})(u_{\mu_2}^*u_{\gamma_2})$. This would be non zero only if there are paths $\nu_1,\nu_2$ of length $q$ such that $\gamma_i=\mu_i\nu_i$ ($i=1,2$) and, in this case, we can write $u_{\mu_1}^*au_{\mu_2}=c\in \calc_0$ and $u_{\nu_i}=u_{\mu_1}^*u_{\gamma_i}$. We  then have
	$$u_{\gamma_1}^*a(\mu_1,\mu_2)u_{\gamma_2}=u_{\nu_1}^*cu_{\nu_2}=0$$ since $s(\nu_1)=s(\gamma_1)\neq s(\gamma_2)=s(\nu_2)$ (using part (2)).
	Summimg up over all $\mu_1,\mu_2$ of length $np$, we get
	$$u_{\gamma_1}^*au_{\gamma_2}=0.$$

\end{proof}

To understand the structure of $q(\cald)$ we start by looking closely at $\mathcal{C}_0$.

We first need some notation.

Recall that we write $q$ for $p-1$.
Since $E$ has no sinks, it allows us, given $v\in E^0$, to find a path that starts in $v$ and has length $q$. Such a path is not unique but we fix one and denote it $\xi_{v}$. It will be fixed throughout the rest of the paper.

For $v\in E^0$, we write
\begin{equation}\label{cv}
	\calc_v=u_{\xi_{v}}u_{\xi_{v}}^*\mathcal{C}_0 u_{\xi_{v}}u_{\xi_{v}}^*\subseteq \mathcal{C}_0.
\end{equation}

\begin{prop}\label{cn}
	We have
	$$\mathcal{C}_n \cong \Sigma_{ |\gamma_1|=|\gamma_2|=np+p-1, s(\gamma_1)=s(\gamma_2)} e_{ \gamma_1 ,\gamma_2 } \otimes \mathcal{C}_{s(\gamma_1)} \subseteq M_{\Gamma_n}(\mathbb{C})\otimes \mathcal{C}_0 $$
	where $\Gamma_n$ is $E^{(n+1)p-1}$, the set of all paths of length $(n+1)p-1$,  $M_{\Gamma_n}$ is the matrix algebra indexed by $\Gamma_n$ and $e_{\cdot,\cdot}$ is a matrix unit there.
	Writing $$\mathcal{A}_n= \Sigma_{ |\gamma_1|=|\gamma_2|=np+p-1, s(\gamma_1)=s(\gamma_2)} e_{\gamma_1 ,\gamma_2} \otimes \mathcal{C}_{s(\gamma_1)},$$ the isomorphism from $\mathcal{C}_n$ to $\mathcal{A}_n$ is given by
	\begin{equation}\label{taun}
		\tau_n(c)= \Sigma_{  |\gamma_i|=np+p-1, s(\gamma_1)=s(\gamma_2)} e_{\gamma_1 ,\gamma_2 } \otimes u_{\xi_{s(\gamma_1)}}u_{\gamma_1}^*cu_{\gamma_2}u_{\xi_{s(\gamma_1)}}^*
	\end{equation}
	for $c\in \mathcal{C}_n$. It's inverse is given by
	\begin{equation}\label{tauninverse}\tau_n^{-1}(e_{\gamma_1 ,\gamma_2  }\otimes b)=u_{\gamma_1}u_{\xi_{s(\gamma_1)}}^*bu_{\xi_{s(\gamma_1)}}u_{\gamma_2}^*
	\end{equation} 	
	for $b\in \mathcal{C}_{s(\gamma_1)}$,  $|\gamma_1|=|\gamma_2|=np+p-1$ and $s(\gamma_1)=s(\gamma_2)$.		
\end{prop}
\begin{proof}
	Note first that, for  $|\gamma_1|=|\gamma_2|=np+p-1$ (with $s(\gamma_1)=s(\gamma_2)$), we can write $\gamma_i=\mu_i\alpha_i$ where $|\mu_i|=np$, $|\alpha_i|=p-1$ and $s(\alpha_1)=s(\alpha_2)$. Then, for $c\in \mathcal{C}_n$ have that $u_{\xi_{s(\gamma_1)}}u_{\gamma_1}^*cu_{\gamma_2}u_{\xi_{s(\gamma_1)}}^*=u_{\xi_{s(\gamma_1)}}u_{\alpha_1}^*u_{\mu_1}^*cu_{\mu_2}u_{\alpha_2}u_{\xi_{s(\gamma_1)}}^*=(u_{\xi_{s(\gamma_1)}}u_{\alpha_1}^*)(u_{\mu_1}^*cu_{\mu_2})(u_{\alpha_2}u_{\xi_{s(\gamma_1)}}^*)$. This lies in $\mathcal{C}_{s(\gamma_1)}$ since each of the three factors lies in $\mathcal{C}_0$ (for the middle one, apply Lemma~\ref{eCnf}(1) successively). So that $\tau_n$ maps $\mathcal{C}_n$ into $\mathcal{A}_n$.
	We now show that $\tau_n$ is a multiplicative map on $\mathcal{C}_n$. For this, we fix $c,b\in \mathcal{C}_n$ and compute
		
	$$\tau_n(c)\tau_n(b)=  \Sigma_{ s(\gamma_1)=s(\gamma_2), |\gamma_i|=np+p-1}\Sigma_{ s(\gamma_1')=s(\gamma_2'), |\gamma'_i|=np+p-1}$$  $$ e_{\gamma_1,\gamma_2 } e_{\gamma_1' ,\gamma_2' }\otimes u_{\xi_{s(\gamma_1)}}u_{\gamma_1}^*cu_{\gamma_2}u_{\xi_{s(\gamma_1)}}^*u_{\xi_{s(\gamma_1')}}u_{\gamma'_1}^*bu_{\gamma'_2}u_{\xi_{s(\gamma_2')}}^*$$ $$=\Sigma_{\gamma_1,\gamma_2'} e_{\gamma_1 ,\gamma_2' } \otimes u_{\xi_{s(\gamma_1)}}u_{\gamma_1}^*c (\Sigma_{\gamma_2} u_{\gamma_2}u_{\gamma_2}^*)bu_{\gamma_2'}u_{\xi_{s(\gamma_1')}}^*=\tau_n(cb).$$
	We used the fact that, in the sum above, $\gamma_2=\gamma_1'$, $s(\gamma_1)=s(\gamma_2)$  and $\Sigma_{\gamma_2} u_{\gamma_2}u_{\xi_{s(\gamma_2)}}^*u_{\xi_{s(\gamma_2)}}u_{\gamma_2}^*=\Sigma_{\gamma_2}u_{\gamma_2}p_{s(\gamma_2)}u_{\gamma_2}^*=\Sigma_{|\gamma_2|=np+p-1}u_{\gamma_2}u_{\gamma_2}^*=I$.
	 
	Now compute, for $c\in \mathcal{C}_n$, $$\tau_n^{-1}(\tau_n(c))=\Sigma_{ |\gamma_1|=|\gamma_2|=np+p-1, s(\gamma_1)=s(\gamma_2)} \tau^{-1}(e_{\gamma_1 ,\gamma_2 }\otimes u_{\xi_{s(\gamma_1)}}u_{\gamma_1}^*cu_{\gamma_2}u_{\xi_{s(\gamma_2)}}^*)$$  $$=\Sigma_{ |\gamma_1|=|\gamma_2|=np+p-1, s(\gamma_1)=s(\gamma_2)} u_{\gamma_1}u_{\xi_{s(\gamma_1)}}^*u_{\xi_{s(\gamma_1)}}u_{\gamma_1}^*cu_{\gamma_2}u_{\xi_{s(\gamma_2)}}^*u_{\xi_{s(\gamma_2)}}u_{\gamma_2}^*$$ $$=\Sigma_{|\gamma_i|=np+p-1, s(\gamma_1)=s(\gamma_2)}u_{\gamma_1}u_{\gamma_1}^*cu_{\gamma_2}u_{\gamma_2}^*=\Sigma_{|\gamma_i|=np+p-1}u_{\gamma_1}u_{\gamma_1}^*cu_{\gamma_2}u_{\gamma_2}^*=c.$$ Note that in the equality before the last one, we used Lemma~\ref{Lempsi}(3) to drop the ``$s(\gamma_1)=s(\gamma_2)$".
	
	And, for the other direction, fix $|\gamma_i|=np+p-1$ with  $s(\gamma_1)=s(\gamma_2)$ and $b\in \mathcal{C}_{s(\gamma_1)}$, and compute
	$$\tau_n(\tau_n^{-1}(e_{\gamma_1 ,\gamma_2  }\otimes b))=\tau_n(u_{\gamma_1}u_{\xi_{s(\gamma_1)}}^*bu_{\xi_{s(\gamma_1)}}u_{\gamma_2}^*)$$  $$= \Sigma_{ |\gamma_i'|=np+p-1} e_{\gamma_1',\gamma_2'} \otimes u_{\xi_{s(\gamma')}}u_{\gamma_1'}^*u_{\gamma_1}u_{\xi_{s(\gamma_1)}}^*bu_{\xi_{s(\gamma_1)}}u_{\gamma_2}^*u_{\gamma_2'}u_{\xi_{s(\gamma_1')}}^*.$$
	Since $u_{\gamma_1'}^*u_{\gamma_1}\neq 0$ only if $\gamma_1=\gamma_1'$ (and then $u_{\gamma_1'}^*u_{\gamma_1}=p_{s(\gamma_1)}$) we have $\gamma_1'=\gamma_1$ and, similarly, $\gamma_2'=\gamma_2$  and we end up with 
	$$e_{\gamma_1 ,\gamma_2 }\otimes u_{\xi_{s(\gamma_1)}}u_{\xi_{s(\gamma_1)}}^*bu_{\xi_{s(\gamma_2)}}u_{\xi_{s(\gamma_2)}}=e_{\gamma_1 ,\gamma_2 }\otimes b.$$
\end{proof}

Now we let $\psi_n:\mathcal{A}_n \rightarrow \mathcal{A}_{n+1}$ be
$$\psi_n=\tau_{n+1} \circ \iota_n \circ \tau_n^{-1}$$ and get

\begin{cor}\label{dirlim}
	$$	q(\cald)\cong \lim (\mathcal{A}_n,\psi_n).$$
\end{cor}


\begin{prop}\label{psin}
	We have
	$$\psi_n(e_{\alpha,\beta}\otimes b)=\Sigma_{|\mu|=p, r(\mu)=s(\alpha)=s(\beta)} e_{\alpha \mu,\beta \mu}\otimes u_{\xi_{s(\mu)}}u_{\mu}^*u_{\xi_{s(\alpha)}}^*bu_{\xi_{s(\beta)}}u_{\mu}u_{\xi_{s(\mu)}}^* $$
	for $b\in u_{\xi_{s(\alpha)}}u_{\xi_{s(\alpha)}}^*\mathcal{C}_0 u_{\xi_{s(\beta)}}u_{\xi_{s(\beta)}}^*$, $s(\alpha)=s(\beta)$ and $|\alpha|=|\beta|=(n+1)p-1$.
\end{prop}
\begin{proof}
	For $b\in u_{\xi_{s(\alpha)}}u_{\xi_{s(\alpha)}}^*\mathcal{C}_0 u_{\xi_{s(\beta)}}u_{\xi_{s(\beta)}}^*$	(with  $|\alpha|=|\beta|=(n+1)p-1$ and $s(\alpha)=s(\beta)$) we have
	
	$$\psi_n(e_{\alpha,\beta}\otimes b)=\tau_{n+1}(u_{\alpha}u_{\xi_{s(\alpha)}}^*bu_{\xi_{s(\alpha)}}u_{\beta}^*)=\Sigma_{|\gamma_i|=(n+2)p-1} e_{\gamma_1,\gamma_2}\otimes u_{\xi_{s(\gamma_1)}}u_{\gamma_1}^*u_{\alpha}u_{\xi_{s(\alpha)}}^*bu_{\xi_{s(\alpha)}}u_{\beta}^*u_{\gamma_2}u_{\xi_{s(\gamma_2)}}^*.$$
	Now, $u_{\alpha}^*u_{\gamma_1}=u_{\mu_1}$ if $\gamma_1=\alpha \mu_1$ (for some $\mu_1$ with $|\mu_1|=p$ and $r(\mu_1)=s(\alpha)$) and $0$ otherwise. Similarly $u_{\beta}^*u_{\gamma_2}=u_{\mu_2}$ if $\gamma_2=\beta \mu_2$ and $0$ otherwise. Thus
	$$\psi_n(e_{\alpha,\beta}\otimes b)=\Sigma_{|\mu_i|=p} e_{\alpha \mu_1,\beta \mu_2}\otimes u_{\xi_{s(\mu_1)}}u_{\mu_1}^*u_{\xi_{s(\alpha)}}^*bu_{\xi_{s(\alpha)}}u_{\mu_2}u_{\xi_{s(\mu_2)}}^*.$$
	By Lemma~\ref{Lempsi}, the terms with $\mu_1\neq\mu_2$ are zero and the result follows.
	
\end{proof}

\
\section{Ideals in $q(\cald)$}

From now on, we assume that $E$ has no sources (and no sinks).

\begin{prop}\label{ideals} Fix $n\geq 0$.
	A subspace $J\subseteq \mathcal{C}_n$ is an ideal there if and only if there is a family of subspaces $\{J_{v}\subseteq \mathcal{C}_{v}\}$ satisfying
	\begin{enumerate}
		\item [(1)] $\tau_n(J)=\sum_{|\mu_i|=np+q, s(\mu_1)=s(\mu_2)} e_{\mu_1,\mu_2}\otimes J_{s(\mu_1)} $ and
		\item[(2)] for every $v\in E^0$, $J_v$ is an ideal in $\calc_v$.
		\item[(3)] For $\mu_1,\mu_2$ with $|\mu_i|=np+q$ and $s(\mu_i)=v$ ($i=1,2$), $J_v=u_{\xi_v}u_{\mu_1}^*Ju_{\mu_2}u_{\xi_v}^*$.
	
	\end{enumerate}
In this case, we have
	\begin{equation}\label{J}
		J=\sum_{|\mu_i|=np+q, s(\mu_1)=s(\mu_2)} u_{\mu_1}u_{\xi_{s(\mu_1)}}^*J_{s(\mu_1)}u_{\xi_{s(\mu_1)}}u_{\mu_2}^* .\end{equation}
\end{prop}
\begin{proof}
Assume that $J\subseteq \calc_n$ is an ideal. Take $c\in J$. Then 
\begin{equation}\label{taunc}
	\tau_n(c)=\sum_{|\mu_1|=|\mu_2|=np+q, s(\mu_1)=s(\mu_2)} e_{\mu_1,\mu_2}\otimes u_{\xi_{s(\mu_1)}}u_{\mu_1}^*cu_{\mu_2}u_{\xi_{s(\mu_2)}}^*.
\end{equation}	
If $|\mu_i|=np+q$ then $u_{\mu_i}u_{\mu_i}^*\in \calc_n$ and, thus, 
$u_{\mu_1}u_{\mu_1}^*cu_{\mu_2}u_{\mu_2}^*	\in J$ and $\tau_n(u_{\mu_1}u_{\mu_1}^*cu_{\mu_2}u_{\mu_2}^*)\in \tau_n(J)$. But $$\tau_n(u_{\mu_1}u_{\mu_1}^*cu_{\mu_2}u_{\mu_2}^*)=\sum_{|\gamma_i|=np+q, s(\gamma_1)=s(\gamma_2)} e_{\gamma_1,\gamma_2}\otimes u_{\xi_{s(\gamma_1)}}u_{\gamma_1}^*u_{\mu_1}u_{\mu_1}^*cu_{\mu_2}u_{\mu_2}^*u_{\gamma_2}u_{\xi_{s(\gamma_2)}}. $$ In this expression the non zero summands are only those for which $\mu_i=\gamma_i$. Thus, for every $|\mu_i|=np+q$ with $s(\mu_1)=s(\mu_2)$,
$$\tau_n(u_{\mu_1}u_{\mu_1}^*cu_{\mu_2}u_{\mu_2}^*)=e_{\mu_1,\mu_2}\otimes u_{\xi_{s(\mu_1)}}u_{\mu_1}^*cu_{\mu_2}u_{\xi_{s(\mu_2)}}^* \in \tau_n(J).$$ 
Write $$J_{\mu_1,\mu_2}:=u_{\xi_{s(\mu_1)}}u_{\mu_1}^*Ju_{\mu_2}u_{\xi_{s(\mu_2)}}^*\subseteq \calc_{s(\mu_1)}.$$ Then
$e_{\mu_1,\mu_2}\otimes J_{\mu_1,\mu_2}\subseteq \tau_n(J)$. Since we also have $\tau_n(J)\subseteq \sum_{|\mu_i|=np+q, s(\mu_1)=s(\mu_2)} e_{\mu_1,\mu_2}\otimes J_{\mu_1,\mu_2}$ (by (\ref{taunc}), we get
\begin{equation}\label{J2}
\tau_n(J)= \sum_{|\mu_i|=np+q, s(\mu_1)=s(\mu_2)} e_{\mu_1,\mu_2}\otimes J_{\mu_1,\mu_2}.
\end{equation}	
Now, for every $\nu_1,\nu_2$ with $|\nu_i|=np+q$ and $s(\nu_1)=s(\nu_2)=s(\mu_1)=s(\mu_2)=v$, we have $e_{\nu_1,\nu_2}\otimes \calc_{v}J_{\mu_1,\mu_2}\calc_{v}=(e_{\nu_1,\mu_1}\otimes \calc_{v})(e_{\mu_1,\mu_2}\otimes J_{\mu_1,\mu_2})(e_{\nu_2,\mu_2}\otimes \calc_{v})\subseteq \tau_n(\calc_n)\tau_n(J)\tau_n(\calc_n)\subseteq \tau_n(J)$. Thus,
\begin{equation}\label{J3} 
\calc_{v}J_{\mu_1,\mu_2}\calc_{v}\subseteq J_{\nu_1,\nu_2}.
\end{equation}
Since $u_{\xi_{s(\mu_i)}}u_{\xi_{s(\mu_i)}}^*\in \calc_{v}$, we find that $J_{\mu_1,\mu_2}=u_{\xi_{s(\mu_1)}}u_{\xi_{s(\mu_1)}}^*J_{\mu_1,\mu_2}u_{\xi_{s(\mu_2)}}u_{\xi_{s(\mu_2)}}^*\subseteq J_{\nu_1,\nu_2}$. By symmetry, $J_{\mu_1,\mu_2}=J_{\nu_1,\nu_2}$ whenever $s(\mu_1)=s(\mu_2)=s(\nu_1)=s(\nu_2)$. Thus we can now write $J_{s(\mu_1)}$ instead of $J_{\mu_1,\mu_2}$. But then, using (\ref{J3}) we get part (2) and (\ref{J2})  implies part (1).	

The converse direction is straightforward and is omitted.

Equation (\ref{J}) follows from part (1) and the definition of $\tau_n^{-1}$.
	\end{proof}

\begin{cor}\label{JnJ0}
	Suppose $J$ is an ideal in $q(\cald)$ and $J_n=J\cap \calc_n$. Assume that for every $v\in E^0$ and $n\geq 0$, $J_{n,v}=J_{0,v}$. Then
	\begin{equation}\label{Jn}
		J_n=\sum_{|\delta_i|=np} u_{\delta_1}J_{0}u_{\delta_2}^*.
	\end{equation}
\end{cor}
\begin{proof} Fix $n\geq 0$.
Every path $\mu$ of length $np+q$ can be written, in a unique way, as $\mu=\delta \alpha$ for paths $\alpha$ and $\delta$ of lengths $q$ and $np$ respectively. Also, given a path $\delta$ of length $np$, there is a path $\alpha$ of length $q$ such that $\delta \alpha$ is well defined ($s(\delta)=r(\alpha)$) and has length $np+q$ (since $E$ has no sources).

It follows from Equation (\ref{J}) that $	J_n=\sum_{|\mu_i|=np+q, s(\mu_1)=s(\mu_2)} u_{\mu_1}u_{\xi_{s(\mu_1)}}^*J_{n,s(\mu_1)}u_{\xi_{s(\mu_1)}}u_{\mu_2}^*$ and, using the discussion above,
$$J_n= \sum_{|\delta_i|=np, |\alpha_i|=q, s(\alpha_1)=s(\alpha_2)} u_{\delta_1}u_{\alpha_1}u_{\xi_{s(\mu_1)}}^*J_{n,s(\mu_1)}u_{\xi_{s(\mu_1)}}u_{\alpha_2}^*u_{\delta_2}^*.$$	
Note that if $r(\alpha_i)\neq s(\delta_i)$, the corresponding summand will vanish so that the sum is really over $\alpha_i,\delta_i$ such that $r(\alpha_i)=s(\delta_i)$. Thus, using the assumption that $J_{n,v}=J_{0,v}$ for every $v$,
$$J_n=\sum_{|\delta_i|=np} u_{\delta_1} (\sum_{|\alpha_i|=q, s(\alpha_1)=s(\alpha_2)} u_{\alpha_1}u_{\xi_{s(\alpha_1)}}^*J_{0,s(\alpha_1)}u_{\xi_{s(\alpha_1)}}u_{\alpha_2}^*)u_{\delta_2}^*=\sum_{|\delta_i|=np}u_{\delta_1} J_0 u_{\delta_2}^*.$$
	
	\end{proof}

In order to discuss the simplicity of the algebra $C^*(E,Z)=\mathcal{T}(X_E,Z)/\mathcal{K}\cong \mathcal{O}(q(\mathcal{D}),q(F))$ we need first the following definition.

\begin{defn}\label{minnonper}
	Let $X$ be a $C^*$-correspondence over a unital $C^*$-algebra $A$. We say that it is \emph{minimal} if there are no non trivial ideals $J\subseteq A$ such that 
	$\langle X, JX \rangle \subseteq J$. It is said to be \emph{nonperiodic} if $X^{\otimes n}$ and $A$ are isometric (that is, there is a unitary map of correspondences from $X^{\otimes n}$ onto $A$) only if $n=0$. 	
	
\end{defn}

Note that, here, minimality means that $q(\cald)$ has no non trivial invariant ideal in the sense of Definition~\ref{idinv} and non periodicity was discussed in Proposition~\ref{nonperiodic}.

The following theorem was proved by J. Schweizer \cite[Theorem 3.9]{Sch01}.

\begin{thm}\label{simpleCP}
	Let $X$ be a full $C^*$-correspondence over a unital $C^*$-algebra $A$. Then the Cuntz-Pimsner algebra $\mathcal{O}(X,A)$ is simple if and only if $X$ is minimal and nonperiodic.
\end{thm}

In order to understand what are the invariant ideals of $q(\cald)$, we first need the following.

Given paths $\alpha,\beta$ of the same length , $s(\alpha)=s(\beta)$ and $r(\alpha)=r(\beta)$, they define a map
$$\pi(\alpha,\beta):\calc_{r(\alpha)}\rightarrow \calc_{s(\alpha)}$$ defined by
\begin{equation}\label{pialphabeta}
	\pi(\alpha,\beta)(c)=u_{\xi_{s(\alpha)}}u_{\alpha}^*u_{\xi_{r(\alpha)}}^*cu_{\xi_{r(\beta)}}u_{\beta}u_{\xi_{s(\beta)}}^*. \end{equation} 

The range of $\pi(\alpha,\beta)$ is contained in $\calc_0$ because of Lemma~\ref{eCnf} (3)  (with $n=0$).

Note that, if $r(\alpha_2)=s(\alpha_1)$ and $r(\beta_2)=s(\beta_1)$ then, for $c\in \calc_{r(\alpha_1\alpha_2)}=\calc_{r(\alpha_1)}$,

 $\pi(\alpha_1\alpha_2,\beta_1\beta_2)(c)=u_{\xi_{s(\alpha_2)}}u_{\alpha_2}^*u_{\alpha_1}^*u_{\xi_{r(\alpha_1)}}^*cu_{\xi_{r(\beta_1)}}u_{\beta_1}u_{\beta_2}u_{\xi_{s(\beta_2)}}^*\\
 =u_{\xi_{s(\alpha_2)}}u_{\alpha_2}^*u_{\xi_{r(\alpha_2)}}^*u_{\xi_{r(\alpha_2)}}u_{\alpha_1}^*u_{\xi_{r(\alpha_1)}}^*cu_{\xi_{r(\beta_1)}}u_{\beta_1}u_{\xi_{r(\beta_2)}}^*u_{\xi_{r(\beta_2)}}u_{\beta_2}u_{\xi_{s(\beta_2)}}^*=\pi(\alpha_2,\beta_2)\pi(\alpha_1,\beta_1)(c)$. 
 
 Thus, in this case,
	\begin{equation}\label{pimult}
		\pi(\alpha_1\alpha_2,\beta_1\beta_2)=\pi(\alpha_2,\beta_2)\pi(\alpha_1,\beta_1).
		\end{equation} 

\begin{prop}\label{invariant}
	Let $J$ be an ideal in $\calc_n$. Then $J$ is invariant (in the sense of Definition~\ref{idinv}) if and only if, for every $e,f\in E^1$ with $r(e)=r(f)$ and $s(e)=s(f)$, $$\pi(e,f)(J_{r(e)})\subseteq J_{s(e)}.$$
\end{prop}
\begin{proof} 
Assume the condition holds. Fix $c\in J$. Then
\begin{equation}
	\tau_n(c)=\sum_{|\mu_1|=|\mu_2|=np+q, s(\mu_1)=s(\mu_2)} e_{\mu_1,\mu_2}\otimes u_{\xi_{s(\mu_1)}}u_{\mu_1}^*cu_{\mu_2}u_{\xi_{s(\mu_2)}}^*.
\end{equation}	
Thus, for every $|\mu_1|=|\mu_2|=np+q$ with $s(\mu_1)=s(\mu_2)$, $u_{\xi_{s(\mu_1)}}u_{\mu_1}^*cu_{\mu_2}u_{\xi_{s(\mu_2)}}^*\in J_{s(\mu_1)}$ and, using the condition of the proposition, for every $g_1,g_2\in E^1$ with $r(g_i)=s(\mu_i)$, $s(g_1)=s(g_2)$ and $r(g_1)=r(g_2)$, we have,
\begin{equation}\label{pi} \pi(g_1,g_2)(u_{\xi_{s(\mu_1)}}u_{\mu_1}^*cu_{\mu_2}u_{\xi_{s(\mu_2)}}^*	)\in J_{s(g_1)}.\end{equation}
Now $c=\sum_{|\mu_i|=np+q, s(\mu_1)=s(\mu_2)}u_{\mu_1}u_{\xi_{s(\mu_1)}}^*u_{\xi_{s(\mu_1)}}u_{\mu_1}^*cu_{\mu_2}u_{\xi_{s(\mu_2)}}^*u_{\xi_{s(\mu_2)}}u_{\mu_2}^*$ and 

 $u_e^*cu_f=\sum_{|\mu_i|=np+q, s(\mu_1)=s(\mu_2)}u_e^*u_{\mu_1}u_{\xi_{s(\mu_1)}}^*u_{\xi_{s(\mu_1)}}u_{\mu_1}^*cu_{\mu_2}u_{\xi_{s(\mu_2)}}^*u_{\xi_{s(\mu_2)}}u_{\mu_2}^*u_f$.
 Thus, for every $\delta_1,\delta_2$ with $|\delta_i|=np+q$ and $s(\delta_1)=s(\delta_2)$, we get
 $$u_{\xi_{s(\delta_1)}}u_{\delta_1}^*u_e^*cu_fu_{\delta_2}u_{\xi_{s(\delta_2)}}^*=\sum_{|\mu_i|=np+q, s(\mu_1)=s(\mu_2)}u_{\xi_{s(\delta_1)}}u_{\delta_1}^*u_e^*u_{\mu_1}u_{\xi_{s(\mu_1)}}^*u_{\xi_{s(\mu_1)}}u_{\mu_1}^*cu_{\mu_2}u_{\xi_{s(\mu_2)}}^*u_{\xi_{s(\mu_2)}}u_{\mu_2}^*u_fu_{\delta_2}u_{\xi_{s(\delta_2)}}^*$$
But $u_{\delta_1}^*u_e^*u_{\mu_1}\neq 0$ only if there is some $g_1\in E^1$ such that $e\delta_1=\mu_1g_1$ and, similarly, there is some $g_2\in E^1$ such that $f\delta_2=\mu_2g_2$ and, when this is the case, $s(g_1)=s(g_2)$, $r(g_1)=r(g_2)$ and
 $$u_{\xi_{s(\delta_1)}}u_{\delta_1}^*u_e^*cu_fu_{\delta_2}u_{\xi_{s(\delta_2)}}^*=\sum_{|\mu_i|=np+q, s(\mu_1)=s(\mu_2)}u_{\xi_{s(\delta_1)}}u_{g_1}^*u_{\xi_{s(\mu_1)}}^*u_{\xi_{s(\mu_1)}}u_{\mu_1}^*cu_{\mu_2}u_{\xi_{s(\mu_2)}}^*u_{\xi_{s(\mu_2)}}u_{g_2}u_{\xi_{s(\delta_2)}}^*.$$
 
But, since $s(\delta_i)=s(g_i)$ and $s(\mu_i)=r(g_i)$ ($i=1,2$), and $u_{\xi_{s(\mu_1)}}u_{\mu_1}^*cu_{\mu_2}u_{\xi_{s(\mu_2)}}^*\in J_{s(\mu_1)}=J_{r(g_1)}$, each summand in this sum lies in $\pi(g_1,g_2)(J_{r(g_1)})\subseteq J_{s(g_1)}=J_{s(\delta_1)}$.  
It follows that
$$\tau_n(u_e^*cu_f)\in\sum_{|\delta_i|=np+q, s(\delta_1)=s(\delta_2)} e_{\delta_1,\delta_2}\otimes J_{s(\delta_1)}=\tau_n(J)$$ and, thus, $u_e^*cu_f\in J$.

For the converse, assume now that $J$ is an invariant ideal in $\calc_n$. Fix $\mu_1,\mu_2$ with $|\mu_i|=np+q$ and $s(\mu_1)=s(\mu_2)$. Then $e_{\mu_1,\mu_2}\otimes J_{s(\mu_1)}$ is contained in $\tau_n(J)$ and, applying $\tau_n^{-1}$, we get
$$u_{\mu_1}u_{\xi_{s(\mu_1)}}^*J_{s(\mu_1)}u_{\xi_{s(\mu_2)}}u_{\mu_2}^*\subseteq J.$$
Then, for $e,f\in E^1$ (with $s(e)=s(f)$ and $r(e)=r(f)$) we have (using the invariance of $J$),
$$u_e^*u_{\mu_1}u_{\xi_{s(\mu_1)}}^*J_{s(\mu_1)}u_{\xi_{s(\mu_2)}}u_{\mu_2}^*u_f\subseteq J.$$
But the left-hand side is $0$ unless there are $\gamma_1,\gamma_2$ such that $\mu_1=e\gamma_1$ and $\mu_2=f\gamma_2$. If this holds then we get (using $s(\gamma_i)=s(\mu_i)$),
$$u_{\gamma_1}u_{\xi_{s(\gamma_1)}}^*J_{s(\gamma_1)}u_{\xi_{s(\gamma_2)}}u_{\gamma_2}^*\subseteq J.$$
Applying $\tau_n$ we get
$$\sum_{|\delta_i|=np+q, s(\delta_1)=s(\delta_2)} e_{\delta_1,\delta_2}\otimes u_{\xi_{s(\delta_1)}}u_{\delta_1}^*u_{\gamma_1}u_{\xi_{s(\gamma_1)}}^*J_{s(\gamma_1)}u_{\xi_{s(\gamma_2)}}u_{\gamma_2}^*u_{\delta_2}u_{\xi_{s(\delta_2)}}^*\subseteq \tau_n(J)$$ and, therefore,
$$u_{\xi_{s(\delta_1)}}u_{\delta_1}^*u_{\gamma_1}u_{\xi_{s(\gamma_1)}}^*J_{s(\gamma_1)}u_{\xi_{s(\gamma_2)}}u_{\gamma_2}^*u_{\delta_2}u_{\xi_{s(\delta_2)}}^*\subseteq J_{s(\delta_1)}$$ for every $\delta_i$ with $|\delta_i|=np+q$ and $s(\delta_1)=s(\delta_2)$. But the left-hand side is $0$ unless $\delta_i=\gamma_i g_i$ for some $g_1,g_2 \in E^1$ (as otherwise either $u_{\delta_1}^*u_{\gamma_1}=0$ or $u_{\gamma_2}^*u_{\delta_2}=0$). Assuming $\delta_i=\gamma_i g_i$ for some $g_1,g_2 \in E^1$ and noting that, in that case, $s(\delta_i)=s(g_i)$ and $s(\gamma_i)=r(g_i)$, we get
\begin{equation}\label{pi1}u_{\xi_{s(g_1)}}u_{g_1}^*u_{\xi_{r(g_1)}}^*J_{r(g_1)}u_{\xi_{r(g_2)}}u_{g_2}u_{\xi_{s(g_2)}}^*\subseteq J_{s(g_1)}.\end{equation}
Since $g_i$ are arbitrary in $E^1$ with $s(g_1)=s(g_2)$ and $r(g_1)=r(g_2)$ (as $\mu_i$ and $\delta_i$ run over all paths of length $np+q$ with $s(\mu_1)=s(\mu_2)$ and $s(\delta_1)=s(\delta_2)$) and the left-hand side of (\ref{pi1}) is $\pi(g_1,g_2)(J_{r(g_1)})$, this completes the proof of this direction.


	\end{proof}


\begin{thm}\label{simplicity}
	Suppose $E$ is finite, irreducible, and has no sources and no sinks. Assume also that condition A(p) holds (and let $p$ be the minimal positive integer for which this holds). Then $C^*(E,Z)$ is simple if and only if the following two conditions hold.
	\begin{enumerate}
		\item [(a)]  $E$ is not a cycle.
		\item[(b)] There is no non trivial closed ideal $J$ in $\calc_0$ that is invariant in the sense of Proposition~\ref{invariant}.
		
		Part (b) is equivalent to
		\item[(b')] There is no family $\{J_{v}\subseteq \calc_{v}\}_{v\in E^0}$ such that
		\begin{enumerate}
			\item [(i)]  For every $v\in E^0$, $J_v$ is an ideal in $\calc_v$ and
			\item[(ii)] for every $e,f\in E^1$ with $s(e)=s(f)$ and $r(e)=r(f)$, we have $\pi(e,f)(J_{r(e)})\subseteq J_{s(e)}.$
		\end{enumerate}
	\end{enumerate}

\end{thm}
\begin{proof} 
We first show that there is a non trivial invariant ideal in $q(\cald)$ if and only if there is a non trivial invariant ideal in $\calc_0$.

Assume first that $J\subseteq q(\cald)$ is a non trivial invariant ideal in $q(\cald)$. For every $n\geq 0$, write $J_n:=J\cap \calc_n$. Then, using Lemma~\ref{eCnf}, $J_n$ is an invariant ideal in $\calc_n$. Since $J=\overline{\cup_n J_n}$ and $J$ is non trivial, there is some $n\geq 0$ such that $J_n$ is non trivial. Fix such $n$. If it is $0$, we are done. In general, we can write
$$\tau_n(J_n)=\sum_{|\mu_i|=np+q} e_{\mu_1,\mu_2}\otimes J_{n, s(\mu_1)}$$ and, for every $e,f\in E^1$ with $s(e)=s(f)$ and $r(e)=r(f)$,
\begin{equation}\label{inv}
	\pi(e,f)(J_{n, r(e)})\subseteq J_{n,s(e)}.\end{equation}
Now set 
$$J_0:=\tau_0^{-1}( \sum_{|\alpha_1|=|\alpha_2|=q, s(\alpha_1)=s(\alpha_2)} 	e_{\alpha_1,\alpha_2}\otimes J_{n,s(\alpha_1)} ).$$
It follows from Proposition~\ref{ideals} that $J_0$ is an ideal in $\calc_0$ and Proposition~\ref{invariant} (together with (\ref{inv})) shows that this ideal is invariant. Since $J_n$ is non trivial, there are $v_1,v_2$ in $s(E^1)$ such that $J_{n,v_1}\neq 0$ and $J_{n,v_2}\neq \calc_{v_2}$ and it follows that  $J_0$ is non trivial. Thus, there is an invariant, non trivial ideal in $\calc_0$.

For the other direction, assume $J_0$ is a non trivial invariant ideal $J_0$ in $\calc_0$ and write (using Proposition~\ref{ideals})
$$\tau_0(J_0)=\sum_{|\alpha_1|=|\alpha_2|=q, s(\alpha_1)=s(\alpha_2)} e_{\alpha_1,\alpha_2}\otimes J_{0,s(\alpha_1)}.$$
Now, for every $n>0$, set
$$J_n:=\tau_n^{-1}(\sum_{|\mu_1|=|\mu_2|=np+q, s(\mu_1)=s(\mu_2)} e_{\mu_1,\mu_2}\otimes J_{0,s(\mu_1)}).$$
By Proposition~\ref{ideals}, $J_n$ is a non trivial ideal in $\calc_n$ and Proposition~\ref{invariant} implies that this ideal is invariant.

We claim that $J_n\subseteq J_{n+1}$. For this, we compute
$$J_n=\sum_{|\mu_1|=|\mu_2|=np+q, s(\mu_1)=s(\mu_2)} \tau_n^{-1}(e_{\mu_1,\mu_2}\otimes J_{0,s(\mu_1)})=\sum_{|\mu_1|=|\mu_2|=np+q, s(\mu_1)=s(\mu_2)} u_{\mu_1}u_{\xi_{s(\mu_1)}}^*J_{0,s(\mu_1}u_{\xi_{s(\mu_2)}}u_{\mu_2}^*$$  $$=\sum_{|\mu_1|=|\mu_2|=np+q, |\gamma_1|=|\gamma_2|=p, s(\mu_1)=s(\mu_2)} u_{\mu_1}u_{\gamma_1}u_{\xi_{s(\gamma_1)}}^*u_{\xi_{s(\gamma_1)}}u_{\gamma_1}^*u_{\xi_{s(\mu_1)}}^*J_{0,s(\mu_1)}u_{\xi_{s(\mu_2)}}u_{\gamma_2}u_{\xi_{s(\gamma_2)}}^*u_{\xi_{s(\gamma_2)}}u_{\mu_2}^*u_{\mu_2}^*$$
The only non zero terms in this sum are those for which $r(\gamma_i)=s(\mu_i)$, $i=1,2$. Hence
$$J_n=\sum_{|\mu_i|=np+q, |\gamma_i|=p, r(\gamma_i)=s(\mu_i), s(\mu_1)=s(\mu_2), s(\gamma_1)=s(\gamma_2)}u_{\mu_1\gamma_1}u_{\xi_{s(\gamma_1)}}^*\pi(\gamma_1,\gamma_2)(J_{0,r(\gamma_1)})u_{\xi_{s(\gamma_2)}}u_{\mu_2\gamma_2}^*$$  $$\subseteq \sum_{|\mu_i|=np+q, |\gamma_i|=p, r(\gamma_i)=s(\mu_i), s(\gamma_1)=s(\gamma_2)}u_{\mu_1\gamma_1}u_{\xi_{s(\gamma_1)}}^*J_{0,s(\gamma_1)}u_{\xi_{s(\gamma_2)}}u_{\mu_2\gamma_2}^*$$  $$=\tau^{-1}_{n+1}(\sum_{\mu_i,\gamma_i}e_{\mu_1\gamma_1,\mu_2\gamma_2}\otimes J_{0,s(\mu_1\gamma_1)})=J_{n+1}$$ proving the claim. Then $J:=\overline{\cup_nJ_n}$ is a non trivial invariant ideal in $q(\cald)$ because if $J=q(\mathcal{D})$, then $I\in J=\overline{\cup_nJ_n}$ and we can find $n$ and $a\in J_n$ such that $\|I-a\|<1$. But then $a\in J_n$ is invertible contradicting the fact that $J_n\neq \mathcal{C}_n$.	

Now, suppose that $E$ is not a cycle. Since we assume that there are no sources. $E$ is not a cycle with an entry (see the definition preceeding Lemma~\ref{cycleentry}). Thus, by Proposition~\ref{nonperiodic} (1), $q(F)$ is nonperiodic. The discussion above shows that it is minimal if and only if there is no non trivial invariant closed ideal in $\calc_0$. It follows from Theorem~\ref{simpleCP} that, if $E$ is not a cycle, $C^*(E,Z)$ is simple if and only if there is no non trivial closed invariant ideal in $\calc_0$.


Now assume $E$ is a cycle. We show that, in this case, $q(F)$ cannot be nonperiodic. Assume that it is nonperiodic. Then, by Proposition~\ref{nonperiodic}, for every $n>0$, there is some path $\alpha$ such that $|\alpha|=n$ and $u_{\alpha}$ is not in $q(\cald)'$. Suppose that the length of the cycle is $m$ and let $n=mp$. We show that for every $\alpha$ with $|\alpha|=n=mp$, $u_{\alpha}\in q(\cald)'$. This will complete the proof. 
Since $n$ is a multiple of $p$ (and we assume condition A(p)), $u_{\alpha}\in (C^*(z))'$. To prove that it commutes with $q(\cald)$, it suffices to show that it commutes with the generators. So, using Lemma~\ref{qDAp}, we fix a generator $u_{\mu}z^lu_{\nu}$ with $|\mu|=|\nu|=t$ and $s(\mu)=s(\nu)$. But, using lemma~\ref{cycle}, we can assume that $\mu=\nu$. Also, since $\{\calc_n\}$ is an increasing sequence, we can assume that $t>n$ and that allows us to write $\mu=\mu_1 \mu_2$ with $|\mu_1|=t-n$ and $|\mu_2|=n$.
Since $|\alpha \mu|=n+t$  we can write $\alpha \mu=\mu'\alpha'$ with $|\mu'|=t$ and $|\alpha'|=n$. But then, since $|\alpha|$ is divisible by $m$, the length of the cycle, $r(\alpha)=s(\alpha)$ and  $r(\mu)=s(\alpha)=r(\alpha)=r(\mu')$. Thus $\mu$ and $\mu'$ have the same length and same endpoint, implying that $\mu=\mu'$ (and, thus, $\alpha \mu=\mu \alpha'$). Then, using Lemma~\ref{uz}, we have $$u_{\alpha}u_{\mu}z^lu_{\mu}^*=u_{\mu}u_{\alpha'}z^lu_{\mu}^*=u_{\mu}z^lu_{\alpha'}u_{\mu}^*=u_{\mu}z^lu_{\alpha'}u_{\mu_2}^*u_{\mu_1}^*.$$
But $\alpha \mu=\mu \alpha'$ and $\mu=\mu_1\mu_2$. Thus $s(\alpha')=s(\mu)=s(\mu_2)$ and, since $|\alpha'|=|\mu_2|$, $\alpha'=\mu_2$ and it follows that $u_{\mu}z^lu_{\alpha'}u_{\mu_2}^*u_{\mu_1}^*=
u_{\mu}z^lu_{\mu_1}^*$. Since $u_{\alpha}u_{\mu_1}u_{\alpha'}=u_{\alpha }u_{\mu}=u_{\mu}u_{\alpha'}$, we get $u_{\alpha}u_{\mu_1}=u_{\mu}$ and $u_{\mu_1}=u_{\alpha}^*u_{\mu}$. Then $u_{\alpha}u_{\mu}z^lu_{\mu}^*=u_{\mu}z^lu_{\mu_1}^*=u_{\mu}z^lu_{\mu}^*u_{\alpha}$. Thus, $u_{\alpha}\in q(\cald)'$.
	\end{proof}

In Theorem~\ref{unweightedsimplicity} we show that, restricting to the unweighted case ($Z_k=I$ for all $k$ and $C^*(E,Z)=C^*(E)$ the Cuntz-Krieger algebra), we get a new proof of a  well known criterion for simplicity of the Cuntz-Krieger algebras.

We now want to describe the gauge-invariant ideals of $C^*(E,Z)$. (See Remark~\ref{gauge}). For this, we first use the analysis of Katsura in \cite{Ka2}. We shall start by introducing some of Katsura's notation and terminology. Here $X$ is a $C^*$-correspondence over the $C^*$-algebra $A$. 

For an ideal $I$ of $A$ we write
$$X(I)=\overline{span}\{\langle \eta,\varphi_X(a)\xi\rangle \in A:\; a\in I, \; \xi,\eta \in X\},$$
$$X^{-1}(I)=\{a\in A :\; \langle \eta,\varphi_X(a)\xi\rangle \in I \;\;\mbox{for all}\;\; \xi,\eta \in X\}.$$ Both are ideals in $A$.
He also defines an ideal $J(I)$, associated to $I$. But, for the case where $\varphi_X(A)=K(X)$ (which is the case where $A=q(\cald)$ and $X=q(F)$, by Lemma~\ref{qF1} (3)), we have $$J(I)=A.$$

\begin{defn} \label{Opair} (\cite[Definition 5.6 and Definition 5.12]{Ka2}) A pair $(I,I')$ of ideals in $A$ is said to be an O-pair if
	\begin{enumerate}
\item[(1)] $X(I)\subseteq I$,
\item[(2)] $J_X\cap X^{-1}(I)\subseteq I$ and
\item[(3)]  $I+J_X\subseteq I' \subseteq J(I)$,
\end{enumerate}
where $J_X:=\varphi_X^{-1}(K(X))\cap (\ker \varphi_X)^{\perp}$ 
\end{defn}

The importance of this concept is shown in the following theorem of Katsura.

\begin{thm}\label{Opairsinvideals}(\cite[Theorem 8.6]{Ka2}) The set of all gauge-invariant ideals of $\mathcal{O}(X,A)$ corresponds bijectively to the set of all $O$-pairs of $X$.
	\end{thm} 

In our case here, $J_{q(F)}=q(\cald)$ and, thus, the O-pairs are pairs of the form $(J,q(\cald))$ where $J$ is an ideal in $q(\cald)$ such that $q(F)(J)\subseteq J$ and $(q(F))^{-1}(J)\subseteq J$.

\begin{defn}\label{finv}
	A fully invariant ideal in $q(\cald)$ is an ideal $J$ such that  $q(F)(J)\subseteq J$ and $(q(F))^{-1}(J)\subseteq J$.
	\end{defn} 

The following now follows from Theorem~\ref{Opairsinvideals}.

\begin{cor}\label{bijfinvideals}
The set of all gauge-invariant ideals of $C^*(E,Z)$ corresponds bijectively to the set of all fully invariant ideals of $q(\cald)$.	
	
\end{cor}

\begin{rem}\label{rembij}
Using \cite[Theorem 8.6]{Ka2} and tracing through Katsura's definitions, we see that, in our special case (where every $O$-pair has the form $(J,q(\mathcal{D}))$ for a fully invariant ideal $J$ in $q(\mathcal{D})$), the bijection in Corollary~\ref{bijfinvideals} is defined as follows. Given a gauge-invariant ideal in $C^*(E,Z)$, the associated fully invariant ideal in $q(\mathcal{D})$ is its intersection with $q(\mathcal{D})$. In the converse direction, given a fully invariant ideal $J$ in $q(\mathcal{D})$, there is a $^*$-representation of $C^*(E,Z) \cong O(q(F),q(\mathcal{D}))$ into the Cuntz-Pimsner algebra $O(q(F)/q(F)J,q(\mathcal{D})/J)$ given by the pair of quotient maps $(q_{q(F)J},q_J)$. The gauge-invariant ideal in $C^*(E,Z)$ associated to $J$ is the kernel of that representation.

	\end{rem} 
\begin{lem}\label{charfullyinv}
	An ideal $J\subseteq q(\cald)$ is fully invariant if and only if , for every $e,f\in E^1$, we have
	\begin{enumerate}
		\item [(i)] $u_e^*Ju_f \subseteq J$ (and, by applying it successively, $u_{\gamma_1}^*Ju_{\gamma_2}\subseteq J$ if $|\gamma_1|=|\gamma_2|$) and 
		\item [(ii)] $u_eJu_f^*\subseteq J$ (and $u_{\gamma_1}Ju_{\gamma_2}^*\subseteq J$ if $|\gamma_1|=|\gamma_2|$) .
	\end{enumerate}
Equivalently, $\bigvee_{e,f \in E^1}u_e^*Ju_f=J$.
\end{lem} 
\begin{proof}
	Suppose $J$ is fully invariant. Then $q(F)^*Jq(F)=q(F)(J)\subseteq J$ and (i) follows from Lemma~\ref{qF}. To prove (ii) note that $$q(F)^*(q(F)Jq(F)^*)q(F)=q(\cald)Jq(\cald)\subseteq J. $$ Thus, every $a\in q(F)Jq(F)^*$ satisfy $q(F)^*aq(F)\subseteq J$. But this means that $a\in q(F)^{-1}(J)$ and, as $J$ is fully invariant, $a\in J$. It follows that $q(F)Jq(F)^*\subseteq J$. In particular, it proves (ii).
	
	For the converse, assume (i) and (ii). Then (i) (using Lemma~\ref{qF}) implies that $q(F)(J)=q(F)^*Jq(F)\subseteq J$. To prove that $q(F)^{-1}(J)\subseteq J$, we need to show that every $a\in q(\cald)$ that satisfies $q(F)^*aq(F)\subseteq J$ lies in $J$. So fix such $a$ . Then, for every $e,f \in E^1$, $u_e^*au_f\in J$. But then, using (ii), we get $u_eu_e^*au_fu_f^*\in J$. Summing over all $e,f \in E^1$ and using Lemma~\ref{alpha}(2) we get $a\in J$.
	
	To prove the last statement, assume first that  $\bigvee_{e,f\in E^1}u_e^*Ju_f=J$. Then (i) clearly holds. To prove (ii), fix $e,f\in E^1$ and note that it suffices to show that, for every $g,h\in E^1$, $u_e(u_g^*Ju_h)u_f^*\subseteq J$ but this is clear since $u_eu_g^*, u_hu_f^* \in q(\cald)$ and $J$ is an ideal. 
	  For the converse, assume that (i) and (ii) hold. We need to show that $J\subseteq \bigvee_{e,f \in E^1} u_e^*Ju_f$. For this, take $a\in J$. Then $a=\sum_{v,w \in E^0} p_vap_w$ and, for $e,f\in E^1$ such that $s(e)=v$ and $s(f)=w$ (recall that such $e,f$ exist since $E$ has no sinks), we have $p_vap_w=u_e^*u_eau_f^*u_f=u_e^*(u_eau_f^*)u_f\in u_e^*Ju_f$ (using (ii)). Thus $a\in \bigvee_{e,f\in E^1}u_e^*Ju_f$.
	
	\end{proof}

\begin{lem}\label{bij}
If $J_1,J_2$ are distinct fully invariant ideals in $q(\cald)$ then $J_1\cap \calc_0$ and $J_2\cap\calc_0$ are distinct ideals in $\calc_0$.		
\end{lem}
\begin{proof}
Assume 	 $J_1,J_2$ are distinct fully invariant ideals in $q(\cald)$. Clearly, $J_1\cap \calc_0$ and $J_2\cap\calc_0$ are ideals in $\calc_0$. We need to show that they are distict. Since $J_i=\overline{\cup_n (\calc_n \cap J_i) }$ and $J_1\neq J_2$ there is some $n$ such that $J_1\cap \calc_n\neq J_2 \cap \calc_n$. Fix this $n$.
Thus, without loss of generality, there is some $a\in J_1\cap \calc_n$ but $a\notin J_2$. For every $\gamma_1,\gamma_2$ with $|\gamma_i|=np$, we have $u_{\gamma_1}^*au_{\gamma_2}\in J_1\cap \calc_0$ (by applying part (i) of Lemma~\ref{charfullyinv}). We will show that, for some $\gamma_i$ of length $np$, $u_{\gamma_1}^*au_{\gamma_2}\notin J_2$ to complete the proof.
But, if this is not the case,  $u_{\gamma_1}^*au_{\gamma_2}\in J_2$ for every such $\gamma_i$ and then, by applying part (ii) of Lemma~\ref{charfullyinv}, $u_{\gamma_1}u_{\gamma_1}^*au_{\gamma_2}u_{\gamma_2}^* \in J_2$ for all such $\gamma_i$. Summing over all these $\gamma_i$s, we get $a\in J_2$ which is a contradiction. 

	\end{proof}

\begin{thm}\label{finvideals}
	Let $J$ be an ideal in $q(\cald)$ and, for $n\geq 0$, write $J_n=J\cap \calc_n$. Then $J$ is fully invariant if and only if the following conditions hold.
	\begin{enumerate} 
		\item[(1)] For every $v\in E^0$, and $n\geq 0$, $$J_{0,v}=J_{n,v},$$
		\item[(2)] for every $\alpha_1,\alpha_2$ with $|\alpha_1|=|\alpha_2|=q$,
		$$u_{\alpha_1}^*J_0u_{\alpha_2}\subseteq J_0$$ and
		\item[(3)] for every $\alpha, \beta$ with $|\alpha|=|\beta|$, $u_{\alpha} J_0 u_{\beta}^*\subseteq J$.
		\end{enumerate}
	\end{thm}
\begin{proof}
Assume first that $J$ is a fully invariant ideal.

Fix $n\geq m$ and $v\in E^0$ and write $q_v:=u_{\xi_v}u_{\xi_v}^*$. Also fix paths $\delta$ and $\gamma$ with $|\gamma|=(n-m)p$, $|\delta|=mp$, $s(\delta)=r(\xi_v)$ and $s(\gamma)=r(\delta)$. Thus $\gamma \delta \xi_v$ is well defined, $|\gamma \delta \xi_v|=np+q$ and $|\delta \xi_v|=mp+q$.

Using the invariance of $J$ (as in Lemma~\ref{charfullyinv}(i) ) and Lemma~\ref{eCnf} (2), we have
$$u_{\gamma}^*J_nu_{\gamma}=u_{\gamma}^*(J\cap \calc_n)u_{\gamma}\subseteq u_{\gamma}^*Ju_{\gamma}\cap u_{\gamma}^*\calc_n u_{\gamma}\subseteq J\cap \calc_m=J_m.$$
Thus
\begin{equation}\label{mn} J_{n,v}=u_{\xi_v}u_{\gamma \delta \xi_v}^*J_nu_{\gamma \delta \xi_v}u_{\xi_v}^*=q_vu_{\delta}^*u_{\gamma}^*J_nu_{\gamma}u_{\delta}q_v \subseteq q_vu_{\delta}^*J_mu_{\delta}q_v=J_{m,v}.\end{equation} 
In particular, $J_{n,v}\subseteq J_{0,v}$. 

To complete the proof of (1), fix  paths $\alpha, \delta$ with $|\alpha|=q$, $|\delta|=np$, $s(\alpha)=v$ and $s(\delta)=r(\alpha)$. Then 
$$ J_{0,v}=u_{\xi_v}u_{\alpha}^*J_0 u_{\alpha}u_{\xi_v}^*=u_{\xi_v}u_{\alpha}^*u_{\delta}^*u_{\delta}J_0 u_{\delta}^*u_{\delta}u_{\alpha}u_{\xi_v}^*=u_{\xi_v}(u_{\alpha}^*u_{\delta}^*)(u_{\delta}J_0 u_{\delta}^*)(u_{\delta}u_{\alpha})u_{\xi_v}^*.$$
By Lemma~\ref{charfullyinv} (ii), $u_{\delta}J_0u_{\delta}^*\subseteq J_n$ and, thus,
$$J_{0,v}\subseteq u_{\xi_v}u_{\delta \alpha}^*J_nu_{\delta \alpha}u_{\xi_v}^*=J_{n,v}$$ completing the proof of (1). Parts (2) and (3) follow from the invariance of $J$ (see Lemma~\ref{charfullyinv} ) and the fact that $u_{\alpha_1}^*\calc_0 u_{\alpha_2}\subseteq \calc_0$.

For the converse, we now assume that (1), (2) and (3) hold.

Fix $a\in J$ and $\alpha_1,\alpha_2 \in E^1$. We will show that $u_{\alpha_1}^*au_{\alpha_2}\in J$.

Since $J=\overline{\cup_n J_n}$, it suffices to assume that $a\in J_n$ for some $n\geq 1$.
For every paths $\gamma,\delta$ with $|\gamma|=|\delta|=np+q-1$, $r(\delta)=s(\alpha_2)$ and $r(\gamma)=s(\alpha_1)$,
we consider $u_{\gamma}u_{\gamma}^*u_{\alpha_1}^*au_{\alpha_2}u_{\delta}u_{\delta}^*$. We claim that it lies in $J$. If $s(\delta)\neq s(\gamma)$, it follows from Lemma~\ref{Lempsi} (3) that this is $0$. So we now assume that $s(\gamma)=s(\delta)=v$. Then $u_{\xi_v}u_{\alpha_1\gamma}^*au_{\alpha_2\delta}u_{\xi_v}^*\in J_{n,v}\subseteq J_{0,v}$ and, thus, there is some path $\mu$ (that depends on $\gamma$ and $\delta$) such that $|\mu|=q$, $s(\mu)=v$ and
$$u_{\xi_v}u_{\gamma}^*u_{\alpha_1}^*a u_{\alpha_2}u_{\delta}u_{\xi_v}^*\in u_{\xi_v}u_{\mu}^*J_0u_{\mu}u_{\xi_v}^*.$$
Multiplying by $u_{\xi_v}^*$ on the left and by $u_{\xi_v}$ on the right yields
$$u_{\gamma}^*u_{\alpha_1}^*a u_{\alpha_2}u_{\delta}\in u_{\mu}^*J_0u_{\mu}$$ and 
$$u_{\gamma}u_{\gamma}^*u_{\alpha_1}^*a u_{\alpha_2}u_{\delta}u_{\delta}^*\in u_{\gamma}u_{\mu}^*J_0u_{\mu}u_{\delta}^*.$$
Write $\gamma=\gamma_1\gamma_2$ and $\delta=\delta_1\delta_2$ where $|\gamma_1|=|\delta_1|=np-1$ and $|\gamma_2|=|\delta_2|=q$ to get
$$u_{\gamma}u_{\mu}^*J_0u_{\mu}u_{\delta}=u_{\gamma_1}(u_{\gamma_2}u_{\mu}^*)J_0(u_{\mu}u_{\delta_2})u_{\delta_1}^*\subseteq u_{\gamma_1}\calc_0J_0\calc_0 u_{\delta_1}^*\subseteq u_{\gamma_1}J_0u_{\delta_1}^*.$$ 	
Using (3), we get
$$u_{\gamma}u_{\gamma}^*u_{\alpha_1}^*a u_{\alpha_2}u_{\delta}u_{\delta}^*\in J,$$ proving the claim.  Summing over all $\gamma$ and $\delta$ as above ($|\gamma|=|\delta|=np+q-1$, $r(\delta)=s(\alpha_2)$ and $r(\gamma)=s(\alpha_1)$), we get	
$$u_{\alpha_1}^*a u_{\alpha_2}\in J.$$

To complete the proof that $J$ is fully invariant we have to show that $u_{\alpha_1}Ju_{\alpha_2}^*\subseteq J$ whenever $\alpha_1,\alpha_2 \in E^1$. So we fix $\alpha_i\in E^1$ and $a\in J_n$. 
For paths $\gamma_1,\gamma_2$ of lengths $np+q+1$, consider $u_{\gamma_1}^*u_{\alpha_1}au_{\alpha_2}^*u_{\gamma_2}$. If this is non zero, there are paths $\nu_1,\nu_2$ such that $\gamma_i=\alpha_i\nu_i$ (so that $|\nu_i|=np+q$ and $s(\nu_i)=s(\gamma_i)$ ) and then $u_{\alpha_i}^*u_{\gamma_i}=u_{\nu_i}$ and $u_{\gamma_1}^*u_{\alpha_1}au_{\alpha_2}^*u_{\gamma_2}=u_{\nu_1}^*au_{\nu_2}$. But this is non zero only if $s(\nu_1)=s(\nu_2)$ (by Lemma~\ref{Lempsi}(3)) so we assume that  $s(\nu_1)=s(\nu_2)=v$ (and then also $s(\gamma_1)=s(\gamma_2)=v$).
Thus, $$u_{\xi_v}u_{\gamma_1}^*u_{\alpha_1}au_{\alpha_2}^*u_{\gamma_2}u_{\xi_v}^*= u_{\xi_v}u_{\nu_1}^*au_{\nu_2}u_{\xi_v}\in J_{n,v}\subseteq J_{0,v} $$ (using (1)). Thus there is a path $\mu$ of length $q$ such that $$u_{\xi_v}u_{\gamma_1}^*u_{\alpha_1}au_{\alpha_2}^*u_{\gamma_2}u_{\xi_v}^*\in u_{\xi_v}u_{\mu}^*J_0u_{\mu}u_{\xi_v}^*.$$ 
It follows that 
$$u_{\gamma_1}^*u_{\alpha_1}au_{\alpha_2}^*u_{\gamma_2}\in u_{\mu}^*J_0u_{\mu}\subseteq J_0$$ (using (2)) and
$$u_{\alpha_1}^*u_{\gamma_1}u_{\gamma_1}^*u_{\alpha_1}au_{\alpha_2}^*u_{\gamma_2}u_{\gamma_2}^*u_{\alpha_2}\in u_{\alpha_1}^*u_{\gamma_1} u_{\mu}^*J_0u_{\mu}u_{\gamma_2}^*u_{\alpha_2}\subseteq u_{\alpha_1}^*J u_{\alpha_2}$$ (using (3)).

Summing up over all $\gamma_1,\gamma_2$ of lengths $np+q+1$, we get $u_{\alpha_1}^*u_{\alpha_1}au_{\alpha_2}^*u_{\alpha_2} \in u_{\alpha_1}^*Ju_{\alpha_2}$ and, as $u_{\alpha_i}u_{\alpha_i}^*\in q(\cald)$,
$$u_{\alpha_1}au_{\alpha_2}^*=u_{\alpha_1}u_{\alpha_1}^*u_{\alpha_1}au_{\alpha_2}^*u_{\alpha_2}u_{\alpha_2}^*\in u_{\alpha_1}u_{\alpha_1}^*Ju_{\alpha_2}u_{\alpha_2}^*\subseteq J.$$
	\end{proof}
\begin{rem}\label{remark2}
Condition (2) in the theorem can be replaced by
\begin{enumerate}
	\item [(2')] For every $\alpha_1,\alpha_2$ with $|\alpha_1|=|\alpha_2|$, $$u_{\alpha_1}^*J_0u_{\alpha_2}\subseteq J_0$$
\end{enumerate}	
Clearly, (2') implies (2) and, thus, (2'), (1) and (3) imply that $J$ is fully invariant. For the converse, if $J$ is fully invariant, then we get (2') by successive application of (i) in Lemma~\ref{charfullyinv}. 
	
\end{rem}


We also have the following.

\begin{thm}\label{finvideals2}
	Let $J$ be an ideal in $q(\cald)$ and, for $n\geq 0$, write $J_n=J\cap \calc_n$. Then $J$ is fully invariant if and only if the following conditions hold.
\begin{enumerate} 
	\item[(1)] For every $v\in E^0$, and $n\geq 0$, $$J_{0,v}=J_{n,v}$$,
	\item[(2)] $J_0$ is an invariant ideal (that is, for every $\alpha, \beta \in E^1$, $u_{\alpha}^*J_0 u_{\beta}\subseteq J_0$)
\end{enumerate}
\end{thm}	
\begin{proof}
	Suppose first that $J$ is fully invariant. Then part (1) follows from Theorem~\ref{finvideals}. Part (2) follows from Lemma~\ref{charfullyinv} (i) and Lemma~\ref{eCnf} (1).
	
	Now assume that (1) and (2) are satisfied. By applying (2) successively we obtain part (2) of Theorem~\ref{finvideals} and it is left to prove part (3) of that theorem. So we fix paths $\alpha_1,\alpha_2$ with $|\alpha_1|=|\alpha_2|=m$. If $m=np$ for some $n$, it follows from Corollary~\ref{JnJ0} that $u_{\alpha_1}J_0u_{\alpha_2}^*\subseteq J_n\subseteq J$ and we are done. In general, we fix $n$ such that $m\leq np$ and paths $\delta_i$ and $\gamma_i$ such that $|\delta_i|=np$ and $\delta_i=\gamma_i \alpha_i$. This is possible because $E$ has no sinks. Then $u_{\delta_i}=u_{\gamma_i}u_{\alpha_i}$ and $u_{\alpha_i}=u_{\gamma_i}^*u_{\delta_i}$. Since $|\delta_i|=np$, it follows from Corollary~\ref{JnJ0} that $u_{\delta_1}J_0u_{\delta_2}^*\subseteq J_n$ and, thus,
	$$u_{\alpha_1}J_0u_{\alpha_2}^*=u_{\gamma_1}^*u_{\delta_1}J_0u_{\delta_2}^*u_{\gamma_2}\subseteq u_{\gamma_1}^*J_nu_{\gamma_2}$$
		and it is left to show that $u_{\gamma_1}^*J_nu_{\gamma_2}\subseteq J_n$, that is, to show that $J_n$ is invariant. We do assume that $J_0$ is invariant and, thus, Proposition~\ref{invariant} holds for $J_0$. But the condition in that proposition (written for $J_n$) depends only on $\{J_{n,v}\}$ and here we assume that $J_{n,v}=J_{0,v}$ for every $v$ so that the invariance of $J_0$ implies the invariance of $J_n$ and this completes the proof.
	
	\end{proof}

In the following proposition, we describe a proceedure to get, from a family $\{J_v: v\in E^0\}$ (where $J_v$ is an ideal in $\mathcal{C}_v$), a fully invariant ideal in $q(\mathcal{D})$.

\begin{prop}\label{calIfinv}
	Fix a family $\{J_v : v\in E^0\}$ where each $J_v$ is an ideal in $\calc_v$ and 
	for every $e,f\in E^1$ with $r(e)=r(f)$ and $s(e)=s(f)$, $$\pi(e,f)(J_{r(e)})\subseteq J_{s(e)}.$$ Then
	\begin{enumerate}
		\item [(1)] The space $$\calj_0:= \sum_{|\alpha|=|\beta|=q, s(\alpha)=s(\beta)=v} u_{\alpha}u_{\xi_{v}}^*J_v u_{\xi_v}u_{\beta}^*$$ is an ideal in $\calc_0$.
		\item[(2)] For every $e,f \in E^1$, $u_e^* \calj_0 u_f\subseteq \calj_0$
		\item[(3)] For every $n\geq 0$, $$\calj_n:=\sum_{|\delta|=|\gamma|=np} u_{\delta}\calj_0 u_{\gamma}^*$$ is an ideal in $\calc_n$.
		\item[(4)] For every $n\geq 0$, $\calj_n\subseteq \calj_{n+1}$ and $$\calj:=\overline{\cup_n \calj_n}$$ is an ideal in $q(\cald)$.
		\item[(5)] $\calj$ is a fully invariant ideal.
	\end{enumerate}
\end{prop} 
\begin{proof}
To prove (1) we need to show that for every $a\in J_v$, $v\in E^0$, $f(z)\in C^*(z)$, paths $\alpha,\beta,\gamma$ and $\delta$ that start in $v$ and satisfy $0\leq |\gamma|=|\delta|\leq q$ and $|\alpha|=|\beta|=q$, we have $x:=u_{\gamma}f(z)u_{\delta}^*u_{\alpha}u_{\xi_v}^*au_{\xi_v}u_{\beta}^*\in \calj_0$.

Write $|\gamma|=|\delta|=m$. For $x$ to be non zero, $\alpha$ can be written as $\alpha=\delta \alpha'$ with $|\alpha'|=q-m$. Also write $\xi_v=\rho_2\rho_1$ with $|\rho_1|=q-m$, $|\rho_2|=m$ and $s(\rho_1)=v$. Then $u_{\xi_v}=u_{\rho_2}u_{\rho_1}$ and $x=u_{\gamma}f(z)u_{\alpha'}u_{\rho_1}^*u_{\rho_2}^*au_{\xi_v}u_{\beta}^*$. Using Lemma~\ref{zcomm} (1), we get
$$x=u_{\gamma \alpha'}u_{\rho_1}^*f(z)u_{\rho_2}^*au_{\xi_v}u_{\beta}^*.$$ Since $u_{\rho_1}^*=u_{\xi_v}^*u_{\rho_2}$, $$x=u_{\gamma \alpha'}u_{\xi_v}^*[u_{\rho_2}f(z)u_{\rho_2}^*a]u_{\xi_v}u_{\beta}^*=u_{\gamma \alpha'}u_{\xi_v}^*[(u_{\xi_v}u_{\xi_v}^*u_{\rho_2}f(z)u_{\rho_2}^*u_{\xi_v}u_{\xi_v}^*)a]u_{\xi_v}u_{\beta}^*.$$ As $u_{\rho_2}f(z)u_{\rho_2}^*\in \calc_0$, we have $u_{\xi_v}u_{\xi_v}^*u_{\rho_2}f(z)u_{\rho_2}^*u_{\xi_v}u_{\xi_v}^*\in \mathcal{C}_v$ 
	and, since $a\in J_v$ (and $J_v$ is an ideal in $\calc_v$), $(u_{\xi_v}u_{\xi_v}^*u_{\rho_2}f(z)u_{\rho_2}^*u_{\xi_v}u_{\xi_v}^*)a\in J_v$ and, thus, $x\in \calj_0$ proving (1).

For (2), take $c\in \calj_0$ and $e,f \in E^1$. It suffices to assume that $$c=u_{\alpha}u_{\xi_v}^*au_{\xi_v}u_{\beta}^*$$ for $a\in J_v$, $s(\alpha)=s(\beta)=v$ and $|\alpha|=|\beta|=q$.
Now, fix paths $\delta_1,\delta_2$ with $|\delta_1|=|\delta_2|=q$ and consider
$$x:=u_{\xi_{s(\delta_1)}}u_{\delta_1}^*u_e^*cu_fu_{\delta_2}u_{\xi_{s(\delta_2)}}^*=u_{\xi_{s(\delta_1)}}u_{\delta_1}^*u_e^*u_{\alpha}u_{\xi_v}^*au_{\xi_v}u_{\beta}^*u_fu_{\delta_2}u_{\xi_{s(\delta_2)}}^*.$$
Assume that $x\ne 0$. Then $u_{e\delta_1}^*u_{\alpha}$ and $u_{\beta}^*u_{f\delta_2}$ are both non zero. Hence there are $g_1,g_2\in E^1$ such that $e\delta_1=\alpha g_1$ and $f\delta_2=\beta g_2$ (in particular, $s(g_i)=s(\delta_i)$ and $r(g_i)=s(\alpha)=s(\beta)=v$). In this case, $u_{\delta_1}^*u_e^*u_{\alpha}=u_{g_1}^*$ and $u_{\beta}^*u_fu_{\delta_2}=u_{g_2}$ and, consequently, 
$$x=u_{\xi_{s(g_1)}} u_{g_1}^*u_{\xi_v}^*au_{\xi_v}u_{g_2}u_{\xi_{s(g_2)}}^*=\pi(g_1,g_2)(a).$$ (See (\ref{pialphabeta})). 
Note also that, as we assume that $x\ne 0$, it follows from Lemma~\ref{Lempsi} (3) (with $n=0$) that $s(\delta_1)=s(\delta_2)$ and, thus, also $s(g_1)=s(g_2)$. Since $a\in J_{r(g_i)}$, we use our assumption to get 
$$x\in J_{s(g_i)}=J_{s(\delta_i)}.$$ 
Finally, we use the fact that $\sum_{|\delta_1|=q} u_{\delta_1}u_{\xi_{s(\delta_1)}}^*u_{\xi_{s(\delta_1)}}u_{\delta_1}^*=I$ (and a similar statement for $\delta_2$) to get
$$u_e^*cu_f=\sum_{|\delta_i|=q}u_{\delta_1}u_{\xi_{s(\delta_1)}}^*(u_{\xi_{s(\delta_1)}}u_{\delta_1}^*u_e^*cu_fu_{\delta_2}u_{\xi_{s(\delta_1)}}^*)u_{\xi_{s(\delta_1)}}u_{\delta_2}^* \in \calj_0.$$

For (3), we fix $n\geq 0$, $x:=u_{\delta}au_{\gamma}^*\in \calj_n$ (where $a\in \calj_0$ and $|\delta|=|\gamma|=np$) and $y:=u_{\alpha}f(z)u_{\beta}\in \calc_n$ (where $f(z)\in C^*(z)$, $np\leq |\alpha|=|\beta|\leq np+q$ and $s(\alpha)=s(\beta)$). We need to show that $xy\in \calj_n$. We write $|\alpha|=|\beta|=np+m$ where $0\leq m \leq q$.

We have $xy=u_{\delta}au_{\gamma}^*u_{\alpha}f(z)u_{\beta}^*$. If $u_{\gamma}^*u_{\alpha}=0$, we are done. Otherwise, $\alpha=\gamma \alpha_2$ (with $|\alpha_2|=m$). We also write $\beta=\beta_1 \beta_2$ with $|\beta_1|=np$ and $|\beta_2|=m$. We get
$xy=u_{\delta}au_{\alpha_2}f(z)u_{\beta_2}^*u_{\beta_1}^*$. Since $u_{\alpha_2}f(z)u_{\beta_2}^*\in \calc_0$ (note that $s(\alpha_2)=s(\alpha)=s(\beta)=s(\beta_2)$) and $a\in \calj_0$, it follows from (1) that $au_{\alpha_2}f(z)u_{\beta_2}^*\in \calj_0$ and $xy\in \calj_n$.

To prove (4) we need to show that, for every paths $\delta,\gamma$ with $|\delta|=|\gamma|=np$, $u_{\delta}\calj_0 u_{\gamma}^*\subseteq \calj_{n+1}$. But $u_{\delta}\calj_0 u_{\gamma}^*=\sum_{|\mu|=|\nu|=p}u_{\delta}u_{\mu}u_{\mu}^*\calj_0 u_{\nu}u_{\nu}^* u_{\gamma}^*$ and it follows from part (2) (applied successively) that $u_{\mu}^*\calj_0 u_{\nu}\subseteq \calj_0$. Thus 
$u_{\delta}\calj_0 u_{\gamma}^*\subseteq \sum u_{\delta \mu}\calj_0 u_{\gamma \nu}^* \subseteq \calj_{n+1}$.
It follows that $\calj_n\subseteq \calj_{n+1}$ and, consequently,  $\calj$ is an ideal in $q(\cald)$.

To prove (5) we shall verify (i) and (ii) of Lemma~\ref{charfullyinv}. 
For (i), we fix $e,f \in E^1$ and $x:=u_{\delta}au_{\gamma}^* \in \calj_n\subseteq \calj$ (where $a\in \calj_0$ and $|\delta|=|\gamma|=np$). Then $u_e^*xu_f=u_e^*u_{\delta}au_{\gamma}^*u_f$ is either $0$ or there are $\delta'$ and $\gamma'$, of length $np-1$ such that $\delta=e\delta'$ and $\gamma=f\gamma'$ and, then,
$$u_e^*xu_f=u_{\delta'}au_{\gamma'}^*=\sum_{g,h \in E^1} u_{\delta'}u_gu_g^*au_hu_h^*u_{\gamma'}^*=\sum_{g,h \in E^1}u_{\delta'g}u_g^*au_hu_{\gamma' h}^*.$$ 
Since $u_g^*au_h\in \calj_0$ (by part (2)), $u_e^*xu_f\in \calj_n\subseteq \calj$.

For (ii) we fix $e,f,a$ and $x$ as above. Since $E$ has no sinks, we can find paths $\mu$ and $\nu$ of length $q$ with $s(\mu)=r(e)$ and $s(\nu)=r(f)$ and, then, $u_e=u_{\mu}^*u_{\mu e}$ and $u_f=u_{\nu}^*u_{\nu f}$. Thus
$$u_exu_f^*=u_{\mu}^*u_{\mu e}u_{\delta}au_{\gamma}^*u_{\nu f}^*u_{\nu}=u_{\mu}^*(u_{\mu e \delta}au_{\nu f \gamma})u_{\nu}.$$
But $|\mu e \delta|=|\nu f \gamma|=(n+1)p$ and, thus, $u_{\mu e \delta}au_{\nu f \gamma}\in \calj_{n+1}\subset \calj$. Using (i) successively, we see that $u_exu_f^*\in \calj$.

	\end{proof}	

\begin{prop}\label{JvequalJv}
	Let $\{J_v : v\in E^0\}$ be a family where each $J_v$ is an ideal in $\calc_v$ and the following two conditions hold.
	\begin{enumerate} 
\item[(H)] For every $e,f\in E^1$ with $r(e)=r(f)$ and $s(e)=s(f)$, $$\pi(e,f)(J_{r(e)})\subseteq J_{s(e)}.$$ 
\item[(S)] For $v\in E^0$, write
$$\mathcal{A}(v):=\cup\{\sum_{|\alpha|=np-q, |\tau|=|\sigma|=q, s(\tau)=s(\sigma)}u_{\xi_v}u_{\alpha}u_{\tau}u_{\xi_{s(\tau)}}^*J_{s(\tau)}u_{\xi_{s(\tau)}}u_{\sigma}^*u_{\alpha}^*u_{\xi_v}^*: \; n\in \mathbb{N}  \}.$$ Then $$\overline{\mathcal{A}(v)}\cap \mathcal{C}_v\subseteq J_v $$ for every $v\in E^0$.


\end{enumerate}
Let $\mathcal{J}$ be the ideal constructed in Proposition~\ref{calIfinv}. Then, for every $v\in E^0$,
$$ \mathcal{J}_v=J_v.$$
\end{prop}
\begin{proof}
Recall first that $$\mathcal{J}_v=u_{\xi_v}u_{\mu_1}^*(\mathcal{J}\cap \mathcal{C}_0)u_{\mu_2}u_{\xi_v}^*$$ for any choice of paths $\mu_1,\mu_2$ of length $q$.

To prove that $J_v\subseteq \mathcal{J}_v$, we fix $a\in J_v$. Then for every paths $\alpha,\beta$ of length $q$ with $s(\alpha)=s(\beta)=v$, we have $u_{\alpha}u_{\xi_v}^*au_{\xi_v}u_{\beta}^* \in \mathcal{J}_0\subseteq \mathcal{J}\cap \mathcal{C}_0$. For $\mu_1,\mu_2$ of length $q$, $u_{\xi_v}u_{\mu_1}^*(u_{\alpha}u_{\xi_v}^*au_{\xi_v}u_{\beta}^* )u_{\mu_2}u_{\xi_v}^*$ is either $0$ (if $\alpha \neq \mu_1$ or $\beta\neq \mu_2$) or is equal to $u_{\xi_v}u_{\xi_v}^*au_{\xi_v}u_{\xi_v}^*=a$. Thus $a\in \mathcal{J}_v$.

Now we turn to prove the other direction. So we fix $v\in E^0$ and $g\in \mathcal{J}_v$. Then there is some $b\in \mathcal{J}\cap \mathcal{C}_0$ such that $g=u_{\xi_v}u_{\xi_v}^*bu_{\xi_v}u_{\xi_v}^*$. Note that $g\in \mathcal{C}_v$. To prove that $g$ lies in $J_v$, we shall show that $g\in \overline{\mathcal{A}(v)}$ and then condition (S) will complete the proof.

Now, fix $\epsilon >0$. Since $b\in \mathcal{J}=\overline{\cup \mathcal{J}_n}$, there are $n\in \mathbb{N}$ and $a\in \mathcal{J}_n$ (that depend on $\epsilon$) such that $\|b-a\|<\epsilon$. But then $$\|g-u_{\xi_v}u_{\xi_v}^*au_{\xi_v}u_{\xi_v}^*\| <\epsilon.$$
Take paths $\alpha,\beta$ of length $np-q$. For every such $\alpha,\beta$,
\begin{equation}\label{ga}\|u_{\alpha}^*u_{\xi_v}^*gu_{\xi_v}u_{\beta}-u_{\alpha}^*u_{\xi_v}^*au_{\xi_v}u_{\beta}\|<\epsilon.
	\end{equation}  (Note that, if $r(\alpha)$ or $r(\beta)$ is different from $v$ , this will be just $0$.)

Since $a\in \mathcal{J}_n$, we can write
$$a=\sum_{|\delta|=|\gamma|=np} u_{\delta} c_{\delta,\gamma}u_{\gamma}^* $$ for some $\{c_{\delta,\gamma}\}\subseteq \mathcal{J}_0$ and then
$$u_{\alpha}^*u_{\xi_v}^*au_{\xi_v}u_{\beta}=\sum_{|\delta|=|\gamma|=np} u_{\alpha}^*u_{\xi_v}^* u_{\delta} c_{\delta,\gamma}u_{\gamma}^*u_{\xi_v}u_{\beta}. $$
But the sum on the right hand side really runs over these $\delta,\gamma$ (of length $np$) that can be written as $\delta=\xi_v \delta'$ and $\gamma=\xi_v \gamma'$ for some $\delta', \gamma'$ (otherwise, either $u_{\xi_v}^*u_{\delta}=0$ or $u_{\gamma}^*u_{\xi_v}=0$). Thus, we can write (renaming $c_{\delta,\gamma}$ as $c_{\delta',\gamma'}$ for simplicity)
$$u_{\alpha}^*u_{\xi_v}^*au_{\xi_v}u_{\beta}=\sum_{|\delta'|=|\gamma'|=np-q} u_{\alpha}^*u_{\xi_v}^* u_{\xi_v}u_{\delta'}c_{\delta',\gamma'}u_{\gamma'}^*u_{\xi_v}^*u_{\xi_v}u_{\beta}$$  $$=\sum_{|\delta'|=|\gamma'|=np-q, r(\delta')=r(\gamma')=v} u_{\alpha}^*u_{\delta'}c_{\delta',\gamma'}u_{\gamma'}^*u_{\beta}=u_{\alpha}^*u_{\alpha}c_{\alpha,\beta}u_{\beta}^*u_{\beta}$$ where the last equality follows from the fact that $u_{\alpha}^*u_{\delta'}=0$ unless $\delta'=\alpha$ and, similarly, $u_{\gamma'}^*u_{\beta}=0$ unless $\gamma'=\beta$. 

It now follows from Equation (\ref{ga}) that, for such $\alpha,\beta$,
\begin{equation}\label{estimate}\|u_{\alpha}^*u_{\xi_v}^*gu_{\xi_v}u_{\beta}-u_{\alpha}^*u_{\alpha}c_{\alpha,\beta}u_{\beta}^*u_{\beta}\|<\epsilon.\end{equation}

Now write, for every path $\alpha$ of length $np-q$, $r_{\alpha}$ to be the projection
$$r_{\alpha}:=u_{\xi_v}u_{\alpha}u_{\alpha}^*u_{\xi_v}^*.$$
Then $\sum_{\alpha} r_{\alpha}=u_{\xi_v}u_{\xi_v}^*$. Since the sum is a projection, the family $\{r_{\alpha}\}$ is orthogonal and, thus, for every family $\{x_{\alpha}\}$ with $x_{\alpha} \in r_{\alpha}q(\mathcal{D})r_{\alpha}$,
$$\| \sum_{\alpha} x_{\alpha} \|=\max_{\alpha}{\|x_{\alpha}\|}.$$

Since $g=u_{\xi_v}u_{\xi_v}^*gu_{\xi_v}u_{\xi_v}^*$,
$g=\sum_{|\alpha|=|\beta|=np-q} r_{\alpha}gr_{\beta}$. But
$ r_{\alpha}gr_{\beta} =u_{\xi_v}u_{\alpha}u_{\alpha}^*u_{\xi_v}^*gu_{\xi_v}u_{\beta}u_{\beta}^*u_{\xi_v}^*$ and, since $g\in \mathcal{C}_0$, it follows from Lemma~\ref{Lempsi} (1) that $r_{\alpha}gr_{\beta}$ is nonzero only if $\alpha=\beta$. Thus
$$g=\sum_{|\alpha|=np-q}r_{\alpha}gr_{\alpha}=\sum_{|\alpha|=np-q}u_{\xi_v}u_{\alpha}u_{\alpha}^*u_{\xi_v}^*gu_{\xi_v}u_{\alpha}u_{\alpha}^*u_{\xi_v}^*.$$

For $\alpha$ of length $np-q$ as above, write $x_{\alpha}:=u_{\xi_v}u_{\alpha}(u_{\alpha}^*u_{\xi_v}^*gu_{\xi_v}u_{\beta}-u_{\alpha}^*u_{\alpha}c_{\alpha,\alpha}u_{\alpha}^*u_{\alpha})u_{\alpha}^*u_{\xi_v}^*$. Then $\|x_{\alpha}\|<\epsilon$ (by Equation (\ref{estimate})) and $x_{\alpha}=r_{\alpha}x_{\alpha}r_{\alpha}$. Thus
$$\|\sum_{|\alpha|=np-q}u_{\xi_v}u_{\alpha}(u_{\alpha}^*u_{\xi_v}^*gu_{\xi_v}u_{\alpha}-u_{\alpha}^*u_{\alpha}c_{\alpha,\alpha}u_{\alpha}^*u_{\alpha})u_{\alpha}^*u_{\xi_v}^* \|<\epsilon.$$

Write $f:=\sum_{|\alpha|=np-q}u_{\xi_v}u_{\alpha}u_{\alpha}^*u_{\alpha}c_{\alpha,\alpha}u_{\alpha}^*u_{\alpha}u_{\alpha}^*u_{\xi_v}^*$. Hence, we get $\|g-f\|<\epsilon$.

We can write $f:=\sum_{|\alpha|=np-q}u_{\xi_v}u_{\alpha}c_{\alpha,\alpha}u_{\alpha}^*u_{\xi_v}^*$ and, noting that $$c_{\alpha,\alpha}\in \mathcal{J}_0=\sum_{|\tau|=|\sigma|=q, s(\tau)=s(\sigma)}u_{\tau}u_{\xi_{s(\tau)}}^*J_{s(\tau)}u_{\xi_{s(\tau)}}u_{\sigma}^*,$$ we can write
$$f=\sum_{|\alpha|=np-q, |\tau|=|\sigma|, s(\tau)=s(\sigma)}u_{\xi_v}u_{\alpha}u_{\tau}u_{\xi_{s(\tau)}}^*d_{\alpha,\tau,\sigma}u_{\xi_{s(\tau)}}u_{\sigma}^*u_{\alpha}^*u_{\xi_v}^* $$ for some $d_{\alpha,\tau,\sigma}$ in $J_{s(\tau)}$.

It follows that $f\in \mathcal{A}(v)$. Therefore, for every $\epsilon >0$, we found $f\in \mathcal{A}(v)$ such that $\|g-f\|<\epsilon$. This shows that $g\in \overline{\mathcal{A}(v)}\cap \mathcal{C}_v$ and, using (S), $g\in J_v$, completing the proof.
	
	\end{proof}	

\begin{rem}\label{remSH} In the next section we show that, in the unweighted case, condition (H) says that the set $W:=\{v\in E^0: \; J_v\neq \{0\}\}$ is hereditary while condition (S) says that $W$ is saturated.
		
	\end{rem} 

We also need the following.

\begin{lem}\label{twoconditions}
	Suppose $J$ is a fully invariant ideal in $q(\mathcal{D})$. Recall (Theorem~\ref{finvideals2}) that for every $v\in E^0$, $J_{n,v}$ is independent of n so we will refer to it simply as $J_v$.
	Then the family $\{J_v\}$ satisfies conditions (H) and (S) of Proposition~\ref{JvequalJv}.
	
	\end{lem}
\begin{proof} 
	Condition (H) follows from Proposition~\ref{invariant}.
	
	To prove condition (S), fix  $v\in E^0$. For every $n\in \mathbb{N}$, paths $\alpha$ (of length $np-q$), and $\tau$, $\sigma$ (of length $q$ with $s(\tau)=s(\sigma)$), and elements $d_{\alpha,\tau,\sigma}$ in $J_{s(\tau)}$, we have $d_{\alpha,\tau,\sigma} \in J_0$, $u_{\tau}u_{\xi_{s(\tau}}^*\in \mathcal{C}_0$ and $u_{\xi_{s(\tau)}}u_{\sigma}^*\in \mathcal{C}_0$. Thus $u_{\tau}u_{\xi_{s(\tau)}}^*d_{\alpha,\tau,\sigma}u_{\xi_{s(\tau)}}u_{\sigma}^*\in J_0$. It follows from Theorem~\ref{finvideals}(3) that $ u_{\xi_v}u_{\alpha}u_{\tau}u_{\xi_{s(\tau)}}^*d_{\alpha,\tau,\sigma}u_{\xi_{s(\tau)}}u_{\sigma}^*u_{\alpha}^*u_{\xi_v}^*\in J$. In fact, it is contained in $u_{\xi_v}u_{\xi_v}^*Ju_{\xi_v}u_{\xi_v}^*$ and, thus, $\mathcal{A}(v)\subseteq u_{\xi_v}u_{\xi_v}^*Ju_{\xi_v}u_{\xi_v}^*$ and also $\overline{\mathcal{A}(v)}\subseteq u_{\xi_v}u_{\xi_v}^*Ju_{\xi_v}u_{\xi_v}^*$. But then 
	$$\overline{\mathcal{A}(v)}\cap \mathcal{C}_0 \subseteq u_{\xi_v}u_{\xi_v}^*Ju_{\xi_v}u_{\xi_v}^* \cap \mathcal{C}_0\subseteq J\cap \mathcal{C}_v=J_v$$
	proving condition (S).

	\end{proof}

\begin{thm}\label{gaugeinvid}
	\begin{enumerate} 
\item[(1)] The construction of Proposition~\ref{calIfinv} gives a bijection between the set of all families $\{J_v\}_{v\in E^0}$ (with $J_v$ an ideal in $\mathcal{C}_v$) that satisfy conditions (S) and (H) and the set of all fully invariant ideals in $q(\mathcal{D})$. 
Given a fully invariant ideal $\mathcal{J}$ in $q(\mathcal{D})$, the associated family is $\{\mathcal{J}_v\}$ where $\mathcal{J}_v=u_{\xi_v}u_{\xi_v}^*(\mathcal{J}\cap \mathcal{C}_0)u_{\xi_v}u_{\xi_v}^*$.
\item[(2)] There is a bijection between the set of all gauge invariant ideals in $C^*(E,Z)$ and the set of all families $\{J_v\}_{v\in E^0}$ that satisfy (H) and (S).

\end{enumerate} 	
	
	\end{thm}
\begin{proof}
Given a family $\{J_v\}$ satisfying (H) and (S), the construction of Proposition~\ref{calIfinv} yields a fully invariant ideal $\mathcal{J}$.

Thus the map $\{J_v\} \mapsto \mathcal{J}$ is well defined into the set of all fully invariant ideals in $q(\mathcal{D})$. 
It follows from Proposition~\ref{JvequalJv} that the map is injective. Surjectivity follows from Lemma~\ref{twoconditions}. This proves part (1).

Part (2) follows from Corollary~\ref{finvideals}. The bijection can be written by composing the bijection described in part (1) with the one described in Remark~\ref{rembij} .	
	
	\end{proof}

\vspace{10mm}


\section{The unweighted case: graph $C^*$-algebras}

In this section we restrict to the unweighted case (that is, $Z=I$). We shall show that our results, restricted to this case, agree with known results for the algebra $C^*(E)$ (which is the algebra $C^*(E,Z)=\calt/\calk$ when $Z=I$).

We will still keep our assumptions that $E$ is finite and has no sources and no sinks. 

Note that, as $Z=I$, Condition A(1) holds so we set $p=1$ (and $q=p-1=0$).

Our main reference for the structure of graph $C^*$-algebras is \cite{Rae}.

Using Lemma~\ref{alpha}, we see that $\{u,p\}$ is a Cuntz-Krieger $E$-family (as in \cite[page 6]{Rae}) and, since $p_v\ne 0$ for every $v\in E^0$, it follows from the Cuntz-Krieger uniqueness theorem (\cite[Theorem 2.4]{Rae}) that $\calt/\calk $ (that is generated by this family) is isomorphic to the graph $C^*$-algebra $C^*(E)$. 

It follows from our definition of $q(\cald)$ that it is the fixed point algebra of the gauge group action defined on $C^*(E)$. Thus, in \cite{Rae}, it is the algebra denoted $C^*(E)^{\gamma}$.

This algebra is shown in \cite[Chapter 3]{Rae} (when $E$ has no sources) to be equal to $\overline{\cup_n \calf_n}$ where 
$$\mathcal{F}_n=\overline{span}\{u_{\alpha}u_{\beta}^* : |\alpha|=|\beta|=n, \; s(\alpha)=s(\beta)\}.$$
Note that $\mathcal{F}_n$ is in fact an algebra and, therefore,
$$\mathcal{F}_n=C^*(\{u_{\alpha}u_{\beta}^* : |\alpha|=|\beta|=n, \; s(\alpha)=s(\beta)\})=\calc_n.$$

In particular, in this case,
\begin{equation}\label{uwC0}
	\mathcal{C}_0=C^*(\{u_{\alpha}u_{\beta}^* :\; |\alpha|=|\beta|=0,\; s(\alpha)=s(\beta)\})=C^*(\{p_v :v\in E^0\})=A.
\end{equation}

Since $q=0$ now, $\xi_v$ is simply $v$ and $u_{\xi_v}$ can be written $p_v$.

For $v\in E^0$, $\calc_v$ (see (\ref{cv})) is now
 $$\calc_v:=p_vAp_v=\mathbb{C}p_v.$$

Recall that by $E^k$ we denote the set of all paths of the length $k$. By $M_{E^k}(\mathbb{C})$ we denote the matrix algebra indexed by the set $E^k$.

Using Proposition~\ref{cn}, we get an isomorphism $\tau_n:\calc_n \rightarrow \cala_n$ where, in our setting now,
$$\cala_n=\sum_{v\in E^0}^{\oplus}( \sum_{|\gamma_1|=|\gamma_2|=n, s(\gamma_i)=v} e_{\gamma_1,\gamma_2}\otimes \mathbb{C}p_v)=\sum_{v\in E^0}^{\oplus}M_{E^n\cap s^{-1}(v)} \otimes \mathbb{C}p_v$$
where $M_{E^n\cap s^{-1}(v)}$ is the algebra of $|{E^n\cap s^{-1}(v)}|\times |{E^n\cap s^{-1}(v)}|$ matrices indexed by the paths of length $n$ that start at $v$. 
In fact 	
\begin{equation}\label{uwtaun}
	\tau_n(c)= \sum_{ v\in E^0, |\gamma_i|=n, s(\gamma_1)=v} e_{\gamma_1 ,\gamma_2 } \otimes u_{\gamma_1}^*cu_{\gamma_2}
\end{equation}
for $c\in \mathcal{C}_n$. It's inverse is given by
\begin{equation}\label{uwtauninverse}\tau_n^{-1}(e_{\gamma_1 ,\gamma_2  }\otimes b)=u_{\gamma_1}bu_{\gamma_2}^*
\end{equation} 	
for $b\in \mathcal{C}_{v}=\mathbb{C}p_v$, $v\in E^0$ and $|\gamma_1|=|\gamma_2|=n$.

Recall that, in \cite{Rae}, $\calf_n(v)$ (for $v\in E^0$) is defined to be $\overline{span}\{u_{\alpha}u_{\beta}^* : |\alpha|=|\beta|=n, s(\alpha)=s(\beta)=v\}$. Then $\tau_n$ maps $\calf_n(v)$ onto $M_{E^n\cap s^{-1}(v)} \otimes \mathbb{C}p_v$.

All this agrees with \cite[Equation (3.8)]{Rae}.

		

Now we let $\psi_n:\mathcal{A}_n \rightarrow \mathcal{A}_{n+1}$ be
 the embedding described in Proposition~\ref{psin}.
Using that proposition, we have, for $e_{\alpha_1,\alpha_2}\otimes p_v$ with $|\alpha_i|=n$ and $s(\alpha_i)=v$,
$$\Psi_n(e_{\alpha_1,\alpha_2}\otimes p_v)=\sum_{|\mu|=1, r(\mu)=v} e_{\alpha_1\mu,\alpha_2\mu}\otimes p_{s(\mu)}.$$
This is consistent with the fact that the inclusion map of $\calf_n$ into $\calf_{n+1}$ does not map $\calf_n(v)$ into $\calf_{n+1}(v)$ but into $\sum^{\oplus}_{w\in s(r^{-1}(v))}\calf_{n+1}(w)$.

For the study of the gauge-invariant ideals we now introduce the following notation. For a subset $W\subseteq E^0$ we write
$$\calf_n(W)=\sum^{\oplus}_{v\in W} \calf_n(v) $$ where $0\leq n <\infty$ and
$$\calf(W)=\bigvee_{n=0}^{\infty} \calf_n(W).$$

We also need the following definition. (See \cite[page 34]{Rae}).

\begin{defn}\label{sather}
\begin{enumerate}
	\item [(1)] A subset $W\subseteq E^0$ is said to be hereditary if whenever $e\in E^1$ with $r(e)\in W$, we also have $s(e)\in W$.
	\item[(2)] A subset $W\subset E^0$ is said to be saturated if whenever $v\in E^0$ and $\{s(e) : e\in E^1,\; r(e)=v\}\subseteq W$, then $v\in W$.
\end{enumerate}	
	
	\end{defn}
In fact, we have the following.

\begin{lem}\label{saturatedn}
	Suppose $W\subseteq E^0$ is saturated and $n>0$. Then whenever $v\in E^0$ and $\{s(\gamma) : \gamma\in E^n,\; r(\gamma)=v\}\subseteq W$, then $v\in W$.
\end{lem}
\begin{proof}
	We proceed by induction on $n$. For $n=1$ it follows from the definition of being saturated. Assume it holds for $k=n-1$. To prove it for $n$, we assume that $v\in E^0$ and 
	\begin{equation}\label{assume}
		\{s(\gamma) : \gamma\in E^n,\; r(\gamma)=v\}\subseteq W. \end{equation}
	We need to show that $v\in W$.
	
	Now, fix a path $\delta$ of length $k$ and $r(\delta)=v$. Write $A(\delta)=\{\alpha \in E^1: r(\alpha)=s(\delta)\}$. For every $\alpha\in A(\delta)$ we can let $\gamma$, of length $n$, be $\gamma:=\delta \alpha$. But then $r(\gamma)=r(\delta)=v$ and it follows by (~\ref{assume}) that $s(\gamma)\in W$. But $s(\gamma)=s(\alpha)$ and, thus, for every $\alpha \in A(\delta)$, $s(\alpha)\in W$. Since $W$ is saturated, $s(\delta)\in W$. This holds for every $\delta$ of length $k$ with $r(\delta)=v$. By the induction assumption, $v\in W$.

	\end{proof}

The following lemma explains Remark~\ref{remSH}.

\begin{lem}\label{HS}
Given a subset $W\subseteq E^0$ and a family $\{J_v\}_{v\in E^0}$ with $J_v=\mathbb{C}p_v$ when $v\in W$ and $J_v=\{0\}$ otherwise, then
\begin{enumerate}
	\item [(1)] $W$ is hereditary if and only if the family $\{J_v\}$ satisfies condition (H) of Proposition~\ref{JvequalJv} and
	\item[(2)] $W$ is saturated if and only if the family $\{J_v\}$ satisfies condition (S) of that proposition.
\end{enumerate}	
		
	\end{lem}
\begin{proof}
	To prove (1), we first assume that $\{J_v\}$ satisfies condition (H). We fix $e\in E^1$ with $r(e)\in W$ (so that $J_{r(e)}=\mathbb{C}p_v$). Condition (H) implies that \begin{equation}\label{H1}\pi(e,e)(p_{r(e)})\in  J_{s(e)}.\end{equation} But, since now $q=0$, $u_{\xi_{s(e)}}=p_{s(e)}$ and $u_{\xi_{r(e)}}=p_{r(e)}$. Thus, we get (using Equation (\ref{pialphabeta}))
	$$\pi(e,e)(p_{r(e)})=p_{s(e)}u_e^*p_{r(e)}p_{r(e)}p_{r(e)}u_ep_{s(e)}=u_e^*u_e=p_{s(e)}$$ and it follows from Equation (\ref{H1}) that $p_{s(e)}\in J_{s(e)}$, so that $J_{s(e)}\ne \{0\}$ and $s(e)\in W$. This shows that $W$ is hereditary.
	
	For the converse, we now assume that $W$ is hereditary and we fix $e,f\in E^1$ with $r(e)=r(f)$ and $s(e)=s(f)$. We want to show that \begin{equation}\label{H2}
		\pi(e,f)(J_{r(e)})\subseteq J_{s(e)}. 
		\end{equation} 
	If $r(e)\notin W$, the left hand side is $\{0\}$ and we are done. So we assume that $r(e)=r(f)\in W$. Since $W$ is hereditary, $s(e)=s(f)\in W$ and we need to show that $\pi(e,f)(p_{r(e)})\in \mathbb{C}p_{s(e)}$. But the left hand side is $p_{s(e)}u_e^*p_{r(e)}u_fp_{s(f)}=u_e^*u_f$ and this is either $0$ (if $e\neq f$) or $p_{s(e)}$ (if $e=f$). In either case, it lies in $\mathbb{C}p_{s(e)}$ and this proves (1).
	
	We now turn to prove part (2) and we note that, for $v\in E^0$, $$\mathcal{A}(v)=\cup_n \{\sum_{|\alpha|=n\geq 1} p_vu_{\alpha}p_{s(\alpha)}J_{s(\alpha)}p_{s(\alpha)}u_{\alpha}^*p_v \}=\cup_n \{\sum_{|\alpha|=n, r(\alpha)=v, s(\alpha)\in W}  \lambda_{\alpha}u_{\alpha}u_{\alpha}^* \}.$$ 
	The equality follows from the definition of $J_v$ and $\lambda_{\alpha}$ are complex numbers that are $0$ if $s(\alpha)\notin W$. 
	
	Suppose first that $W$ is saturated and fix $v\in E^0$. We want to show that $\mathcal{C}_v\cap \overline{\mathcal{A}(v)}\subseteq J_v$. So we fix $a\in \mathcal{C}_v\cap \overline{\mathcal{A}(v)}$. If $a=0$, we are done. Otherwise, since $\mathcal{C}_v=\mathbb{C}p_v$, we can, and will, assume that $a=p_v$. So that $p_v\in \overline{\mathcal{A}(v)}$. Thus, there is some $n\in \mathbb{N}$ and numbers $\{\lambda_{\alpha}\}_{|\alpha|=n,r(\alpha)=v} $ where $\lambda_{\alpha}=0$ whenever $s(\alpha)\notin W$ such that
	$$\|p_v-\sum_{|\alpha|=n, r(\alpha)=v} \lambda_{\alpha}u_{\alpha}u_{\alpha}^* \|<1/2 .$$
	But $p_v=\sum_{|\alpha|=n, r(\alpha)=v}u_{\alpha}u_{\alpha}^*$ and $\{u_{\alpha}u_{\alpha}^*\}$ is an orthogonal family of projections. It follows that, for every $\alpha$ of length $n$ with $r(\alpha)=v$, we have $|1-\lambda_{\alpha}|<1/2$. Thus, for every such $\alpha$, $\lambda_{\alpha}\neq 0$ and, therefore, for every $\alpha$ of length $n$ with $r(\alpha)=v$, $s(\alpha)\in W$. Since $W$ is saturated, it follows from Lemma~\ref{saturatedn} that $v\in W$, so that $p_v\in J_v$. This proves one direction.
	
	For the other direction, assume that condition (S) holds. Thus, for every $v\in E^0$, $\mathcal{C}_v\cap \overline{\mathcal{A}(v)}\subseteq J_v$. We want to show that $W$ is saturated. So we fix $v\in E^0$ such that, for every $e\in E^1$ with $r(e)=v$ we have $s(e)\in W$. Then $p_v=\sum_{e\in E^1, r(e)=v}u_eu_e^*$ but $p_v$ clearly lies in $\mathcal{C}_v$ and the equality shows that it lies in $\mathcal{A}(v)$. By the assumption (condition (S)), $p_v\in J_v$ so that $v\in W$. This shows that $W$ is saturated, completing the proof.

	\end{proof}

The following theorem is known (see Theorem 4.14 and Remark 4.16 in \cite{Rae} and also \cite{CK}). Here we show that it follows from our Theorem~\ref{simplicity}.

\begin{thm}\label{unweightedsimplicity}
For a finite graph $E$ with	no sinks and no sources we have the following. $C^*(E)$ is simple if and only if $E$ is transitive (i.e., for every two vertices $w,v$, there is a path that starts in $v$ and ends in $w$) and $E$ is not a single cicle.	
	\end{thm} 
\begin{proof}
	Note that condition (b') of Theorem~\ref{simplicity} is equivalent, in the unweighted case (using Lemma~\ref{HS}(1)), to the statement that $E^0$ has no non trivial hereditary subsets. But this is equivalent to $E$ being transitive. To see this, assume first that  $E$ is transitive and $W\subseteq E^0$ is a non trivial hereditary set. We can find $w\in W$ and $v\notin W$. By transitivity, there is a path $\alpha$ with $r(\alpha)=w$ and $s(\alpha)=v$ but this will contradict the assumption that $W$ is hereditary.
	
	In the other direction, assume that $E^0$ has no non trivial hereditary subsets and fix $w\in E^0$. Write $H(w):=\{v\in E^0: \mbox{there is a path  } \gamma \mbox{  with  }  s(\gamma)=v, r(\gamma)=w \}$. Then $H(w)$ is hereditary and contains $w$. Thus it is equal $E^0$, showing that $E$ is transitive.
	
	\end{proof}

Applying Theorem~\ref{gaugeinvid} and Lemma~\ref{HS} we get the following theorem. It presents another proof  of \cite[Theorem 4.1 (a)]{BPRS}. However, note that we are assuming that $E$ is a finite graph with no sinks or sources.

\begin{thm}\label{bijunweighted}
For a finite graph $E$ with	no sinks and no sources we have the following.
\begin{enumerate}
	\item [(1)] There is a bijection between the collection of all saturated and hereditary subsets $W$ of $E^0$ and the collection of all gauge-invariant ideals of $C^*(E)$.
	\item[(2)] There is a bijection between the collection of all saturated and hereditary subsets $W$ of $E^0$ and the collection of all fully invariant ideals in $C^*(E)^{\gamma}$.
	\item[(3)] The bijection in part (2) is given by the maps
	$$ W\mapsto \mathcal{F}(W),\;\;\; J\mapsto \{v\in E^0:\; p_v\in J\}.$$
	
\end{enumerate}
\begin{proof}
Parts (1) and (2) follow from Theorem~\ref{gaugeinvid} and Lemma~\ref{HS}. For (3), take $W$ that is saturated and hereditary. Lemma~\ref{HS} associates with it the family $\{J_v\}$ where $J_v=\mathbb{C}p_v$ if $v\in W$ and $J_v=\{0\}$ otherwise. Applying the construction of Proposition~\ref{calIfinv} to this family, we first get 
$$\mathcal{J}_0=\sum_{v\in E^0}p_vJ_vp_v=span\{p_v: v\in W\}.$$	
Then, for $n\geq 0$, 
$$\mathcal{J}_n=\sum_{|\delta|=|\gamma|=n}u_{\delta}\mathcal{J}_0 u_{\gamma}^*=\sum_{v\in W}\mathcal{F}_n(v)=\mathcal{F}(W)$$ and, finally, 
$$\mathcal{J}=\overline{\cup_n \mathcal{F}_n(W)}=\mathcal{F}(W) .$$
	
	\end{proof}

	\end{thm}

\section{Example: A cycle}

In this section we study the case where $E$ is a directed cycle. To simplify the notation we assume that it is a cycle of length $3$ whose vertices are $v_1,v_2,v_3$ and the edges are $e_1,e_2,e_3$ satisfying $s(e_i)=v_i$ and $r(e_i)=v_{i+1}$. (Here and throughout this section, when we add indices of vertices, it is addition modulo $3$).

In the unweighted case (when the algebra is $C^*(E)$), $C^*(E)$ is not simple (\cite[Theorem 4.14]{Rae}) but it follows from \cite[Theorem 4.1 (a)]{BPRS} (see Theorem~\ref{bijunweighted}) that it has no non trivial gauge-invariant ideals since 
$E^0$ has no non trivial subsets that are hereditary and saturated. We shall show here that, in the weighted case, $C^*(E,Z)$ can have non trivial gauge-invariant ideals.

As was noted in the proof of Lemma~\ref{cycleentry}, the set $\{\delta_{\gamma} : |\gamma|=k\}$ is a basis for $X_{E^k}$ (viewed as a vector space) and $\{\delta_{\gamma}: 0\leq |\gamma|<\infty\}$ is a basis for the Fock correspondence $\mathcal{F}(X_E)$. Recall also that we show there that each $\delta_{\gamma}$ is an eigenvector for each of the operators $Z, S_{\alpha}Z^lS_{\alpha}^*$.

It will be convenient, in this section, to use a different notation for the $\delta_{\gamma}$s. Since every path now is uniquely given by its starting point and its length, we write $\alpha_{n,i}$ ($i\in \{1,2,3\}, n\geq 0$) for the path that starts at $v_i$ and has length $n$. We also write $\delta_{n,i}$ for $\delta_{\alpha_{n,i}}$. Thus, $\{\delta_{k,i}: i=1,2,3 \}$ is a basis for $X_{E^k}$.

Note also that $\langle \delta_{k,i},\delta_{k,j}\rangle(v_l)=0$ unless $l=i=j$ and, if $l=i=j$, we get $1$. Thus, if $i\neq j$, $\langle \delta_{k,i},\delta_{k,j}\rangle =0$ and $\langle \delta_{k,i},\delta_{k,i}\rangle=\delta_{v_i}\in C(E^0)$. It follows that $\|\delta_{k,i}\|^2=\|\langle \delta_{k,i},\delta_{k,i}\rangle\|=\|\delta_{v_i}\|=1$.

Since $\{\delta_{n,i}\}$ are eigenvectors of $Z$ and, for a fixed $k$, $\{\delta_{k,i}\}$ is a basis of eigenvectors for $Z_k$, to fix $Z_k$, it suffices to fix the eigenvalues.

So, now we fix $3$ positive numbers $\{t_1,t_2,t_3\}$ that are not all equal to each other. Thus there is a pair of different numbers, say $t_1\neq t_3$, and one of them is necessarily different from $1$, say $t_3\neq 1$. Using these numbers, we set $Z_k$ to be such that $Z_k\delta_{k,i}=\delta_{k,i}$ if $k$ is even and $Z_k\delta_{k,i}=t_i\delta_{k,i}$ if $k$ is odd.

It follows that, for $k\geq 0$, $Z_{k+2}=I_2\otimes Z_k$ (where $I_2$ is the identity on $X_{E^2}$) so that the weight sequence $Z$ satisfies condition A(2).

To study ideals in $C^*(E,Z)$ for this choice of $E$ and $Z$, we start by defining the following characters on $q(\mathcal{D})$. To do this we first write $\mathcal{Q}_0$ for the $^*$-algebra generated by $\{u_{\alpha}z^lu_{\beta}^*: |\alpha|=|\beta|, l\geq 0, s(\alpha)=s(\beta) \}=\{u_{\alpha}z^lu_{\alpha}^* : 0\leq |\alpha|<\infty, l\geq 0\}$ (where the equality follows from the fact that there is only one path with a given starting point and a given length). It follows from Lemma~\ref{qDAp} that $\mathcal{Q}_0$ is a dense sub $^*$-algebra of $q(\mathcal{D})$.

\begin{defn}\label{phini}
For $n\geq 0$ , $i\in \{1,2,3\}$ and $a\in \mathcal{Q}_0$, we define $\phi_{n,i}(a)$ as follows. Let $A\in q^{-1}(a)\subseteq \mathcal{D}$ and define $\phi_{n,i}(a)$ by
$$\phi_{n,i}(a)\delta_{v_i} =\lim_{m\rightarrow \infty} \langle A\delta_{n+6m,i},\delta_{n+6m,i}\rangle .$$
\end{defn}

\begin{prop}
For $n\geq 0$ and $i\in \{1,2,3\}$, $\phi_{n,i}$ is a well defined linear and multiplicative functional on $\mathcal{Q}_0$ of norm $1$.
\end{prop}
\begin{proof}
To prove that $\phi_{n,i}(a)$ is well defined (for $a\in \mathcal{Q}_0$), we need to show that the limit in Definition~\ref{phini} exists and is independent of the choice of $A$.

Write $\mathcal{D}_0$ for the $^*$-subalgebra of $\mathcal{D}$ which is generated by $\{S_{\alpha}Z^lS_{\alpha}^* : 0\leq |\alpha|<\infty, l\geq 0\}$. Clearly every $a\in \mathcal{Q}_0$ has some $A\in \mathcal{D}_0 $ with $q(A)=a$. We shall first show that, for this choice of $A$, the limit $\lim_{m\rightarrow \infty} \langle A\delta_{n+6m,i},\delta_{n+6m,i}\rangle$ exists. In fact, we show that, for such $A$, $\langle A\delta_{n+6m,i},\delta_{n+6m,i}\rangle$ is independent of $m$ for $m$ large enough.

In fact, it will be enough to show it for every generator of $\mathcal{D}_0$. So we fix a generator $A=S_{\alpha_{k,j}}Z^lS_{\alpha_{k,j}}^*$ and claim that 
\begin{equation}\label{SZS}
	S_{\alpha_{k,j}}Z^lS_{\alpha_{k,j}}^*\delta_{n+6m,i}=\left\{ \begin{array}{cc} 0 & i+n+6m\neq k+j (\mod 3) \mbox{ or } n+6m<k \\ \delta_{n+6m,i} & n+6m-k \mbox{ is even and non negative} \\ t_i^l\delta_{n+6m,i} & n+6m-k \mbox{ is odd and non negative}. \end{array}\right.
\end{equation}
Indeed, for this to be non zero, we need that $\alpha_{n+6m,i}=\alpha_{k,j}\gamma$ for some $\gamma$ and this holds if and only if $n+6m\geq k$ and $n+6m+i=r(\alpha_{n+6m,i})=r(\alpha_{k,j})=k+j$ (modulo $3$). If this holds, then $\gamma=\alpha_{n+6m-k,i}$ and $	S_{\alpha_{k,j}}Z^lS_{\alpha_{k,j}}^*\delta_{n+6m,i}=	S_{\alpha_{k,j}}Z^l\delta_{n+6m-k,i}$. Now recall that $Z\delta_{s,i}=\delta_{s,i}$ if $s$ is even and $Z\delta_{s,i}=t_i\delta_{s,i}$ if $s$ is odd to complete the proof of Equation (~\ref{SZS}).

It is now clear that, for every generator $A=	S_{\alpha_{k,j}}Z^lS_{\alpha_{k,j}}^*$ of $\mathcal{D}_0$, $A\delta_{n+6m,i}=\lambda_{n,i}\delta_{n+6m,i}$ for some number $\lambda_{n,i}$ which is independent of $m$ for $m$ large enough ($6m>k-n$). Thus, for $m$ large enough,
$$\langle A\delta_{n+6m,i},\delta_{n+6m,i}\rangle=\lambda_{n,i}\delta_{v_i} $$ and the limit in Definition~\ref{phini} exists for every $A\in \mathcal{D}_0$. 

Now fix some $a\in \mathcal{Q}_0$. Then there is some $A_0\in \mathcal{D}_0$ such that $q(A_0)=a$. Given any $A\in q^{-1}(a)$, we have $A-A_0\in K(\mathcal{F}(X_E))$ and, thus, $\|Q_n(A-A_0)Q_n\|\rightarrow 0$ (Lemma~\ref{K}). Since $\delta_{n+6m,i}=Q_{n+6m}\delta_{n+6m,i}$,
$\lim_{m\rightarrow \infty} \langle (A-A_0)\delta_{n+6m,i},\delta_{n+6m,i}\rangle=0$ and
$$\lim_{m\rightarrow \infty} \langle A\delta_{n+6m,i},\delta_{n+6m,i}\rangle =\lim_{m\rightarrow \infty} \langle A_0\delta_{n+6m,i},\delta_{n+6m,i}\rangle +\lim_{m\rightarrow \infty} \langle (A-A_0)\delta_{n+6m,i},\delta_{n+6m,i}\rangle $$  \begin{equation}\label{lim} =\lim_{m\rightarrow \infty} \langle A_0\delta_{n+6m,i},\delta_{n+6m,i}\rangle .\end{equation} It follows that $\phi_{n,i}(a)$ is well defined for every $a\in \mathcal{Q}_0$.

Since linearity is clear, $\phi_{n,i}$ defines a linear functional on $\mathcal{Q}_0$. It is also multiplicative since $\delta_{n+6m,i}$ are eigenvectors for every $A\in \mathcal{D}_0$.  

Since, for every $A_0\in \mathcal{D}_0$, $\langle A_0\delta_{n+6m,i},\delta_{n+6m,i}\rangle=\lambda_{n,i}\delta_{v_i} $ where $A_0\delta_{n+6m,i}=\lambda_{n,i}\delta_{n+6m,i}$ (so that $|\lambda_{n,i}|\leq \|A_0\|$), we have $\|\lim_{m\rightarrow \infty} \langle A_0\delta_{n+6m,i},\delta_{n+6m,i}\rangle\|\leq \|A_0\|$ (note that $\|\delta_{v_i}\|=1$).
It follows from Equation (~\ref{lim}) that this holds for every $A$ with $q(A)=q(A_0)$. Since, for $a\in \mathcal{Q}_0$, $\|a\|=\inf\{\|A\|: A\in \mathcal{D}, q(A)=a\}$, we have $|\phi_{n,i}(a)|\leq \|a\|$.

As $\mathcal{Q}_0$ is dense in $q(\mathcal{D})$, we can extend $\phi_{n,i}$ (uniquely) to a character of $q(\mathcal{D})$.

	
	\end{proof}
\begin{example}\label{ZZ}
Recall that $Z\delta_{s,i}=\delta_{s,i}$ if $s$ is even and $Z\delta_{s,i}=t_i\delta_{s,i}$ if $s$ is odd. Thus  $Z\delta_{n+6m,i}=\delta_{n+6m,i}$ if $n$ is even and $Z\delta_{n+6m,i}=t_i\delta_{n+6m,i}$ if $n$ is odd. It follows that $\phi_{n,i}(z)=1$ if $n$ is even and $\phi_{n,i}(z)=t_i$ if $n$ is odd.  	
	
	\end{example}

Now we write $K_{n,i}$ for the kernel of $\phi_{n,i}$ and observe the following.

\begin{lem}\label{invariance}
	Fix $n\geq 0$ and $i\in \{1,2,3\}$. 
	\begin{enumerate}
		\item [(1)] $K_{n,i}\cap \mathcal{Q}_0$ is dense in $K_{n,i}$.
		\item[(2)] For every $e,f\in E^1$, $u_eK_{n,i}u_f^*\subseteq K_{n+1,i}$.
		\item[(3)] For every $e,f\in E^1$, $u_e^*K_{n,i}u_f\subseteq K_{n-1,i}$.
	\end{enumerate}
\end{lem}
\begin{proof}
	To prove (1) let $a$ be in $K_{n,i}$ and find a sequence $a_j \rightarrow a$ with $a_j\in \mathcal{Q}_0$. Then $\phi_{n,i}(a_j)\rightarrow \phi_{n,i}(a)=0$ and, thus, $a_j-\phi(a_j)I \rightarrow a$. Since $a_j-\phi_{n,i}(a_j) \in K_{n,i}\cap \mathcal{Q}_0$, we are done.
	
	For (2), we fix $e,f\in E^1$. Using (1) it is enough to show that $u_e(K_{n,i}\cap \mathcal{Q}_0)u_f^*\subseteq K_{n+1,i}$. So we fix $a\in K_{n,i}\cap \mathcal{Q}_0$. There is some $A\in \mathcal{D}_0$ such that $a=q(A)$ and 
	$$\lim_{m\rightarrow \infty} \langle A\delta_{n+6m,i},\delta_{n+6m,i}\rangle =0.$$ 
	We have $u_eau_f^*=q(S_eAS_f^*)$ and we need to show that $$0=\lim_{m\rightarrow \infty} \langle S_eAS_f^*\delta_{n+1+6m,i},\delta_{n+1+6m,i}\rangle =\lim_{m\rightarrow \infty} \langle AS_f^*\delta_{n+1+6m,i},S_e^*\delta_{n+1+6m,i}\rangle.$$ 
	But for this limit to be non zero we should have $S_f^*\delta_{n+1+6m,i}\neq 0$ and $S_e^*\delta_{n+1+6m,i}\neq 0$ and, in such a case, $S_f^*\delta_{n+1+6m,i}=\delta_{n+6m,i}=S_e^*\delta_{n+1+6m,i}$ and we get $\phi_{n+1,i}(u_eau_f^*)=\phi_{n,i}(a)=0$, proving (2). The proof of (3) is similar and is omitted.

	\end{proof}

\begin{cor}\label{examplecycle}
	The ideal $K_1:=\cap_n K_{n,1}$ is a non trivial fully covariant ideal in $q(\mathcal{D})$ and gives rise to a non trivial gauge-invariant ideal in $C^*(E,Z)$.
\end{cor}
\begin{proof}
It follows from Lemma~\ref{invariance} that, for every $e,f \in E^1$, both $u_eK_1u_f^*$ and $u_e^*K_1u_f$ are contained in $K_1$. By Lemma~\ref{charfullyinv}, $K_1$ is fully invariant. 

To show that it is non trivial, we use the fact that $\phi_{n,1}(z)$ is either $1$ or $t_1$ (see Example~\ref{ZZ}), so that $\phi_{n,1}((z-I)(z-t_1I))=0$ for every $n$, and it follows that $(z-I)(z-t_1I)\in K_1$. Since $\phi_{n,3}(z)=t_3$ when $n$ is odd and $t_3$ is different from both $t_1$ and $1$, we see that $\phi_{n,3}((z-I)(z-t_1I))\neq 0$ for odd $n$ so that $(z-I)(z-t_1I)\neq 0$. Thus $K_1\neq \{0\}$.

A similar argument shows that $z-t_3I$ is not in $K_1$ so that $K_1$ is a proper ideal.
	
	\end{proof}

\begin{rem}
Here $E$ is a cycle of length $k=3$ and the weight sequence $Z$ has period $p=2$. But one can apply similar arguments for $k>3$ to find weight sequences $Z$ (with period $p>1$) such that $C^*(E,Z)$ has a non trivial gauge-invariant ideal. In the computation above one should replace $6$, that appears in the formulas, with $kp$.
\end{rem}

\end{document}